\documentclass[12pt,a4paper,reqno]{amsart}
\usepackage[noadjust]{cite}
\usepackage{amsfonts}
\usepackage{amsthm}
\usepackage{amsmath}
\usepackage{amscd}
\usepackage[applemac]{inputenc}
\usepackage{t1enc}
\usepackage[mathscr]{eucal}
\usepackage{indentfirst}
\usepackage[dvipsnames]{xcolor}
\usepackage{graphicx}
\usepackage{graphics}
\usepackage{pict2e}
\usepackage{epic}
\numberwithin{equation}{section}
\usepackage[section]{placeins}
\usepackage[margin=2.9cm]{geometry}
\usepackage{epstopdf}
\usepackage{hyperref} 

\theoremstyle{plain}
\newtheorem{Th}{Theorem}[section]
\newtheorem{Lemma}[Th]{Lemma}
\newtheorem{Corollary}[Th]{Corollary}
\newtheorem{Proposition}[Th]{Proposition}

\theoremstyle{definition}
\newtheorem{Definition}[Th]{Definition}

\newtheorem{Remark}[Th]{Remark}
\newtheorem{?}[Th]{Problem}

\newcommand{\barC}{{\overline{C}}}

\begin{document}

\title{Properties of Clifford Legendre Polynomials}

\author[H. Baghal Ghaffari, J.A. Hogan, J.D. Lakey]{Hamed Baghal Ghaffari, Jeffrey A. Hogan, Joseph D. Lakey}
	
	

	

\begin{abstract}
Clifford-Legendre and Clifford-Gegenbauer polynomials are eigenfunctions of  certain differential operators acting on functions defined on $m$-dimensional euclidean space ${\mathbb R}^m$ and taking values in the associated Clifford algebra ${\mathbb R}_m$.  New recurrence and Bonnet type formulae for these polynomials are proved, as  their Fourier transforms are computed. Explicit representations in terms of spherical monogenics and Jacobi polynomials are given, with consequences including the interlacing of the zeros. In the case $m=2$ we describe a degeneracy between the even- and odd-indexed polynomials.
\end{abstract}

\bigskip
\maketitle
\section{Introduction}
In one dimension, the  Legendre polynomials $\{P_n\}_{n=0}^\infty$ have been long been used as spectral elements in schemes for solving differential equations, and  are used in the numerical computation of  prolate spheroidal wave functions. It's important in this context that the Legendre polynomials satisfy certain recurrence relations, and in particular the Bonnet formula which expresses the product $xP_n(x)$ as a linear combination of $P_{n+1}(x)$ and $P_{n-1}(x)$ -- see e.g. \cite{hogan2011duration,boyd2005algorithm}. This property allows for the reduction of the computation of prolate spheroidal wavefunctions to the computation of the eigenvectors and eigenvalues of a tri-diagonal matrix.

Clifford algebras are generalizations to higher dimensions of the normed division algebras of real and complex numbers and the quaternions; they are normed algebras where the multiplication is non-commutative, but still associative \cite{gradshteyn2007ryzhik}. With a view to constructing higher dimensional prolate spheroidal wavefunctions, here we investigate Clifford algebra-valued polynomials on the $m$-dimensional euclidean space ${\mathbb R}^m$ known as Clifford-Gegenbauer polynomials and the special case of Clifford-Legendre polynomials. These polynomials take values in the $2^m$-dimensional Clifford algebra ${\mathbb R}_m$.
Orthogonal polynomials in Clifford analysis were introduced in \cite{cnops1989orthogonal} and their properties and applications were demonstrated in \cite{brackx2001generalized,de2006multi}.

 The key result of the paper is the Bonnet type formula for the Clifford-Legendre polynomials. We give explicit representations of these polynomials as products of Jacobi polynomials in the radial direction and spherical monogenics. We compute the Fourier transform of the restriction of Clifford-Legendre polynomials to the unit ball in ${\mathbb R}^m$ and demonstrate that despite the absence of a higher-dimensional Sturm-Liouville theory,  the zeros of the Clifford-Legendre polynomials are (radially) interlaced. 

In the second section, we give background related to Clifford algebra and Clifford analysis. In section 3, we define the Clifford-Gegenbauer polynomials and investigate some of their properties. In Section 4, we compute the Fourier transform of the restrictions of the Clifford-Legendre Polynomials to balls and use this to provide a suitable normalization. Plots of some normalized Clifford-Legendre Polynomials are then provided. In section 5, we investigate some properties of the Clifford-Legendre differential equation and will prove that the radial part of the Clifford-Legendre polynomials are Jacobi polynomials and that their zero sets are interlaced. Finally, in the last section we obtain Bonnet type formula for Clifford-Legendre polynomials.

\section{Background}
Let
$\mathbb{R}^{m}$
be 
$m$-dimensional 
euclidean space and let
$\{e_{1},e_{2},\dots e_{m}\}$
be an orthonormal basis for
$\mathbb{R}^{m}.$
We endow these vectors with the multiplicative properties
\begin{eqnarray*}
e_{j}^{2}&=&-1,\; \; j=1,\dots , m,\\
e_{j}e_{i}&=&-e_{i}e_{j}, \;\; i\neq j, \;\; i,j=1,\dots , m.
\end{eqnarray*}
For any subset
$A=\{j_{1},j_{2},\dots, j_{h}\}\subseteq \{1,\dots ,	m\}=M,$ with $j_1<j_2<\cdots <j_h$
we consider the formal product
$e_{A}=e_{j_{1}}e_{j_{2}}\dots e_{j_{h}}.$
Moreover for the empty set
$\emptyset$
one puts
$e_{\emptyset}=1$ (the identity element). The Clifford algebra ${\mathbb R}_m$ is then the $2^m$-dimensional algebra 
$${\mathbb R}_m=\bigg\{\sum\limits_{A\subset M}\lambda_Ae_A:\, \lambda_A\in{\mathbb R}\bigg\}.$$
Every element $\lambda =\sum\limits_{A\subset M}\lambda_Ae_A\in{\mathbb R}_m$ may be decomposed as  
$\lambda=\sum\limits_{k=0}^{m}[\lambda]_{k},$
where 
$[\lambda]_{k}=\sum\limits_{\vert A\vert=k}\lambda_{A}e_{A}$
is the so-called 
$k$-vector
part of 
$\lambda\, (k=0,1,\dots ,m).$

Denoting by 
$\mathbb{R}_{m}^{k}$
the subspace of all 
$k$-vectors
in
$\mathbb{R}_{m},$
i.e., the image of 
$\mathbb{R}_{m}$
under the projection operator 
$[\cdot]_{k},$
one has the multi-vector structure decomposition
$\mathbb{R}_{m}=\mathbb{R}_{m}^{0}\oplus \mathbb{R}_{m}^{1}\oplus\cdots \oplus \mathbb{R}_{m}^{m},$
leading  to the identification of
$\mathbb{R}$
with the subspace of real scalars
$\mathbb{R}_{m}^{0}$
and of
$\mathbb{R}^{m}$
with the subspace of real Clifford vectors 
$\mathbb{R}_{m}^{1}.$ The latter identification is achieved by identifying the point
$(x_{1},\dots,x_{m})\in{\mathbb R}^m$
with the Clifford number
$x=\sum\limits_{j=1}^{m}e_{j}x_{j}\in{\mathbb R}_m^1$.
The Clifford number 
$e_{M}=e_{1}e_{2}\cdots e_{m}$
is called the pseudoscalar; depending on the dimension 
$m,$
the pseudoscalar commutes or anti-commutes with the 
$k$-vectors
and squares to 
$\pm 1.$
The Hermitian conjugation is the real linear mapping $\lambda\mapsto\bar{\lambda}$ of ${\mathbb R}_m$ to itself satisfying
\begin{eqnarray*}
	\overline{\lambda \mu}&=&\bar{\mu}\bar{\lambda},\;\;\;\; \textnormal{for all}\;\lambda,\mu\in\mathbb{R}_{m}\\
	\overline{\lambda_{A}e_{A}}&=&\lambda_{A}\overline{e_{A}},\;\;\; \lambda\in\mathbb{R},\\
	\overline{e_{j}}&=&-e_{j},\;\; j, \;\; j=1,\cdots , m.
\end{eqnarray*}
The Hermitian conjugation leads to a Hermitian inner product and its associated norm on 
$\mathbb{R}_{m}$
given respectively by
$$(\lambda, \mu)=[\bar{\lambda}\mu]_{0}\;\;\;\textnormal{and}\;\;\; \vert\lambda\vert^{2}=[\bar{\lambda}\lambda]_{0}=\sum\limits_{A}\vert\lambda_{A}\vert^{2}.$$
The product of two vectors splits up into a scalar part and a 2-vector, also called a bivector:
$$xy=-\langle x, y\rangle +x\wedge y$$
where
$\langle x,y\rangle=-\sum\limits_{j=1}^{m}x_{j}y_{j}\in \mathbb{R}^{0}_{m}$,
and
$x\wedge y=\sum\limits_{i=1}^{m}\sum\limits_{j=i+1}^{m}e_{i}e_{j}(x_{j}y_{j}-x_{j}y_{i})\in\mathbb{R}^{2}_{m}$.
Note that the square of a vector variable 
$x$
is scalar-valued and equals the norm squared up to minus sign:
$$x^{2}=-\langle x,x\rangle=-\vert x\vert^{2}.$$
Clifford analysis offers a function theory which is a higher-dimensional analogue of the theory of holomorphic functions of one complex variable. The functions considered are defined in the Euclidean space 
$\mathbb{R}^{m}$
and take their values in the Clifford algebra 
$\mathbb{R}_{m}.$

The central notion in Clifford analysis is monogenicity, which is a multidimensional counterpart of holomorphy in the complex plane. 

Let $\Omega\subset{\mathbb R}^m$, $f:\Omega\to{\mathbb R}^m$ and $n$ a non-negative integer. We say $f\in C^n(\Omega,{\mathbb R}_m )$ if $f$ and all its partial derivatives of order less than or equal to $n$ are continuous.
\begin{Definition} Let $\Omega\subset{\mathbb R}^m$. 
A function 
$f\in C^1(\Omega ,{\mathbb R}_m)$
is said to be left monogenic in that region if 
$$\partial_{x}f=0.$$
Here 
$\partial_{x}$
is the Dirac operator in
$\mathbb{R}^{m}$, i.e.,
$$\partial_{x}=\sum\limits_{j=1}^{m}e_{j}\partial_{x_{j}},$$
where
$\partial_{x_{j}}$
is the partial differential operator
$\dfrac{\partial}{\partial x_{j}}.$
The Euler operator is defined on  
$C^1(\Omega ,\mathbb{R}^{m})$
by
$$E=\sum\limits_{j=1}^{m}x_{j}\partial_{x_{j}}.$$
If $k$ is a non-negative integer and $f\in C^1({\mathbb R}^m\setminus\{0\},{\mathbb R}^m)$ is homogeneous of degree $k$ (i.e., $f(\lambda x)=\lambda^kf(x)$ for all $\lambda >0$ and $x\in{\mathbb R}^m$) then $Ef=kf$.
The Laplace operator is factorized by the Dirac operator as follows:
\begin{equation}
\Delta_{m}=-\partial_{x}^{2}.
\end{equation}
\end{Definition}
The notion of right monogenicity is defined in a similar way by letting the Dirac operator act from the right. It is easily seen that if a Clifford algebra-valued function 
$f$
is left monogenic, its Hermitian conjugate 
$\bar{f}$
is right monogenic. 

\begin{Th}(\textbf{Clifford-Stokes theorem})\label{Clifford-Stokes theorem}
	Let $\Omega\subset{\mathbb R}^m$
	$f,g\in {C}_{1}(\Omega)$
	And assume that 
	$C$
	is a compact orientable 
	$m-$dimensional manifold with boundary
	$\partial(C).$
	Then for each
	$C\subset \Omega,$
	one has
	$$\int\limits_{\partial C}f(x)n(x)g(x)d\sigma(x)=\int\limits_{C}[(f(x)\partial_{x})g(x)+f(x)(\partial_{x}g(x))]dx.$$
	where 
	$n(x)$
	is the outward-pointing unit normal vector on
	$\partial C.$
\end{Th}
\begin{proof}
	See the proof at
	\cite{delanghe2012clifford}
	\end{proof}
\begin{Definition}\label{left monogenic homogeneous polynomial}
A left monogenic homogeneous polynomial
$Y_{k}$
of degree 
$k\; (k\geq 0)$
on
$\mathbb{R}^{m}$
is called a left solid inner spherical monogenic of order 
$k.$
The set of all left solid inner spherical monogenics of order 
$k$
will be denoted by
$M_{l}^{+}(k).$
It can be shown 
\cite{delanghe2012clifford}
that the dimension of 
$M_{l}^{+}(k)$
is given by 
$$\dim M_{l}^{+}(k)=\frac{(m+k-2)!}{(m-2)!k!}=\binom{m+k-2}{k}=d_{k},$$
We may choose an orthonormal basis 
for each 
$M_{l}^{+}(k)$, $(k\geq 0)$ i.e., a collection $\{Y_{k}^{j}\}_{j=1}^{d_{k}}$ which spans $M_l^+(k)$ and for which 
$$\int\limits_{S^{m-1}}\overline{Y_{k}^{j}(\theta)}Y_{k}^{j'}(\theta)d\theta=\delta_{jj'}.$$
\end{Definition}
\begin{Remark}\label{Remark2.3}
By direct calculation we see that
\begin{eqnarray*}
e_{j}x&=&-2x_{j}-xe_{j},\\
\partial_{x} E&=&\partial_{x}+E\partial_{x}.
\end{eqnarray*}
\end{Remark}
\begin{Lemma}\label{lem: P to Q}
If
$Y_{k}(x)\in M_{l}^{+}(k),$
and
$P_{m}$
is a
polynomial of degree
$m,$
of a single variable,
then
\begin{eqnarray*}
\partial_{x}[xP_{m}(\vert x\vert^{2})Y_{k}(x)]&=&Q_{m}(\vert x\vert^{2})Y_{k}(x),\\
\partial_{x}[P_{m}(\vert x\vert^{2})Y_{k}(x)]&=&xQ_{m-1}(\vert x\vert^{2})Y_{k}(x)
\end{eqnarray*}
where $Q_m$ and $Q_{m-1}$ are polynomials of degree $m$ and $m-1,$ of a single variable, respectively.
\end{Lemma}	
\begin{Remark}
It's important to note that although the polynomial 
$P_{m}$ in Lemma \ref{lem: P to Q}
has degree 
$m$, the polynomial
$P_{m}(\vert x\vert^{2})$
has degree
$2m.$
\end{Remark}	
\begin{Remark}\label{Remark2.4}
If 
$f:\mathbb{R}^{m}\to\mathbb{R}_{m}$
we define
$Qf$
by
$Qf(x)=xf(x).$
Then we have that
\begin{eqnarray*}
\partial_{x}Q&=&-mI-Q\partial_{x}-2E,\\
\partial_{x}^{2}Q&=&Q\partial_{x}^{2}-2\partial_{x}.
\end{eqnarray*}
\end{Remark}
\begin{Remark}\label{remark1.5}
If $\{A_{j}\}_{j=1}^{\infty},\, \{B_{j}\}_{j=1}^{\infty}$ are sequences of real numbers for which
$ B_{j}=A_{j+1}-A_{j}$
and
$d=B_{j+1}-B_{j}$,
is constant, then the general term of
$ \{ A_{n} \} $
is given by
$$ A_{n}=A_{1}+(n-1)B_{1}+ \frac{(n-1)(n-2)}{2}d. $$
\end{Remark}
We will need also the following lemma which is easy to obtain by direct calculation.
\begin{Lemma}\label{lem: D and Delta on Y_k}
For
$Y_{k}\in M_{l}^{+}(k)$
and 
$s\in \mathbb{N}$
the following fundamental formulas hold:
\begin{equation}
\partial_{x}[x^{s}Y_{k}(x)]=
\left\lbrace \begin{array}{l}
-sx^{s-1}Y_{k}(x)\hspace*{3.55cm} \textnormal{for}\;s\;\textnormal{even},\\
-(s+2k+m-1)x^{s-1}Y_{k}(x)\;\;\;\;\;\; \textnormal{for}\;s\;\textnormal{odd.}\\
\end{array} \right.
\end{equation}
and for $s\geq 2,$
\begin{equation}
\Delta_{m}[x^{s}Y_{k}(x)]=
\left\lbrace \begin{array}{l}
-s(s+2k+m-2)x^{s-2}Y_{k}(x)\hspace*{1.75cm} \textnormal{for}\;s\;\textnormal{even},\\
-(s+2k+m-2)(s-1)x^{s-2}Y_{k}(x)\;\;\;\;\;\; \textnormal{for}\;s\;\textnormal{odd.}\\
\end{array} \right.
\end{equation}
\end{Lemma}
\bigskip
\section{Clifford-Gegenbauer and Clifford-Legendre Polynomials}
For $r>0$, let $B(r)$ be the ball of radius $r$ and centre $0$ in ${\mathbb R}^m$, i.e.,
$$B(r)=\{x\in{\mathbb R}^m:\, |x|\leq r\}.$$
The class of $k$-times continuously differentiable functions $f:B(r)\to{\mathbb R}_m$ is denoted $C^k(B(r),{\mathbb R}_m)$.
\begin{Definition}
Given $\alpha>-1$, let $D_{\alpha}$
be the differential operator defined on 
$C^1(B(1),\mathbb{R}_{m})$
by
\begin{equation}
D_{\alpha}f(x)=(1+x^{2})^{-\alpha}\partial_{x}((1+x^{2})^{\alpha+1}f(x)).
\end{equation}
\end{Definition}	
\begin{Definition}
Let
$\alpha>-1$
and let
$Y_{k}^{i}\in M_{l}^{+}(k)$
be fixed where 
$i\in\{1,2,3,\cdots ,d_k\}$
and
$m,n\in \mathbb{N}$.
Then the \textbf{Clifford-Gegenbauer polynomial}, 
$C_{n,m}^{\alpha}(Y_{k}^{i})(x),$ 
is defined by
\begin{equation}\label{equation3.2}
C_{n,m}^{\alpha}(Y_{k}^{i})(x)=D_{\alpha}D_{\alpha+1}\cdots D_{\alpha+n-1}Y_{k}^{i}(x).
\end{equation}
\end{Definition}
The following description of the Clifford-Gegenbauer polynomials is a generalization of classical Rodrigues' formula for Gegenbauer polynomials on the line.

\begin{Th}(Rodrigues' Formula)\label{th: rodrigues}
The Clifford Gegenbauer polynomials 
$C_{n,m}^{\alpha}(Y_{k}^{i})(x)$
are also determined by 
\begin{equation}\label{rodrigues}
C_{n,m}^{\alpha}(Y_{k}^{i})(x)=(1+x^{2})^{-\alpha}\partial_{x}^{n}((1+x^{2})^{\alpha+n}Y_{k}^{i}(x)).
\end{equation}
\end{Th}
\begin{proof}
See the proof at 
\cite{delanghe2012clifford}.
\end{proof}
As a consequence of the Rodrigues' formula (\ref{rodrigues}), we have that the Clifford-Gegenbauer polynomials are eigenfunctions of a differential operator as below.
\begin{Th}(\textbf{Differential equation for the Gegenbauer polynomials})\label{differentialforClifford}
	For all
	$n,k\in\mathbb{N}$
	and
	$\alpha>-1$, the Clifford-Gegenbauer polynomial $C_{n,m}^{\alpha}(Y_{k}^{i})(x)$ is an eigenfunction of the differential operator $D_\alpha\partial_x$
	with real eigenvalue 
	$C(\alpha,n,m,k)$, i.e.,
	$$D_{\alpha}\partial_{x}(C_{n,m}^{\alpha}(Y_{k}^{i})(x))=C(\alpha,n,m,k)C_{n,m}^{\alpha}(Y_{k}^{i})(x),$$
	or
	\begin{equation}\label{equation3.5}
	(1+x^{2})^{-\alpha}\partial_{x}^{2}(C_{n,m}^{\alpha}(Y_{k}^{i})(x))-2(\alpha+1)x\partial_{x}(C_{n,m}^{\alpha}(Y_{k}^{i})(x))=C(\alpha,n,m,k)C_{n,m}^{\alpha}(Y_{k}^{i})(x),
	\end{equation}
where
$$C(\alpha,n,m,k)=\begin{cases}
n(2\alpha+n+m+2k)&\text{ if $n$ is even}\\
(2\alpha+n+1)(n+m+2k-1)&\text{ if
$n$ is odd.}
\end{cases}$$
\end{Th}
\begin{proof}
	For the proof, see
	\cite{delanghe2012clifford}
\end{proof}
	\begin{Definition}\label{Definition of Gegenbauer}
The Clifford-Legendre polynomial $C_{n,m}^{0}(Y_{k}^{i})(x)$ is the special case of Clifford-Gegenbauer polynomial $C_{n,m}^\alpha(Y_k^i)(x)$ that arises when
$\alpha=0$. From Theorem \ref{th: rodrigues} we have
$$C_{n,m}^{0}(Y_{k}^{i})(x)=\partial_{x}^{n}((1-\vert x\vert^{2})^{n}Y_{k}^{i}(x)).$$
\end{Definition}
\begin{Lemma}\label{EulerofCliffordLegendre}
Let
$C_{n,m}^{0}(Y_{k}^{i})(x)$
be a Clifford-Legendre polynomial and 
$E$
be the Euler operator. Then 
$$E[C_{n,m}^{0}(Y_{k}^{i})(x)]=(n+k)C_{n,m}^{0}(Y_{k}^{i})(x)-2n\partial_{x}(C_{n,m}^{0}(Y_{k}^{i})(x)).$$
\end{Lemma}	
\begin{proof}
An application of the binomial theorem gives
\begin{equation}
E[C_{n,m}^{0}(Y_{k}^{i})(x)]=E[\partial_{x}^{n}(1-\vert x\vert^{2})^{n}Y_{k}^{i}(x)]
=\sum\limits_{r=0}^{n}{n\choose r}(-1)^{r}E\partial_{x}^{n}[\vert x\vert^{2r}Y_{k}^{i}(x)]. \label{C-L homog}
\end{equation}
Note that if $g$ is homogeneous of degree $k$ (i.e., $g(\lambda x)=\lambda^{k}g(x)$ for $\lambda>0$) then
$ E(g(x))=kg(x).$
By Lemma \ref{lem: D and Delta on Y_k} we have
\begin{eqnarray*}
\partial_{x}[\vert x\vert^{2r}Y_{k}^{i}(x)]&=&2r\big(x\vert x\vert^{2r-2}Y_{k}^{i}(x)\big)\\
\partial_{x}^{2}[\vert x\vert^{2r}Y_{k}^{i}(x)]&=&-2r(m+2r+2k-2)\vert x\vert^{2r-2}(Y_{k}^{i})(x).
\end{eqnarray*}
We conclude that $\partial_x[|x|^{2r}Y_k^i(x)]$ is homogeneous of degree $2r+k-1$ and $\partial_{x}^{2}[\vert x\vert^{2r}Y_{k}^{i}(x)]$
is  homogeneous of degree
$2r+k-2$.
An inductive argument gives us that 
$\partial_{x}^{n}[\vert x\vert^{2r}Y_{k}^{i}(x)]$
is homogeneous of degree
$2r-n+k$ so that 
\begin{equation}
E\partial_{x}^{n}[\vert x\vert^{2r}Y_{k}^{i}(x)]=(2r-n+k)\partial_{x}^{n}[\vert x\vert^{2r}Y_{k}^{i}(x)].\label{homog 2}
\end{equation}
Applying (\ref{homog 2}) to (\ref{C-L homog}) yields
\begin{align}
E\partial_{x}^{n}[C_{n,m}^{0}(Y_{k}^{i})(x)]
&=\sum\limits_{r=0}^{n}{n\choose r}(-1)^{r}(2r+k-n)\partial_{x}^{n}[\vert x\vert^{2r}Y_{k}^{i}(x)]\notag\\
&=\partial_{x}^{n}\bigg[\sum\limits_{r=0}^{n}{n\choose r}(-1)^{r}2r\vert x\vert^{2r}Y_{k}^{i}(x)\bigg]\notag\\
&+(k-n)\partial_{x}^{n}\bigg[\sum\limits_{r=0}^{n}{n\choose r}(-1)^{r}\vert x\vert^{2r}Y_{k}^{i}(x)\bigg].\label{homog 3}
\end{align}
However, note that differentiating the equation $(1-t^{2})^{n}=\sum\limits_{r=0}^{n}{n\choose r}(-1)^{r}t^{2r}$ with respect to $t$ gives
$-2nt(1-t^{2})^{n-1}=\sum\limits_{r=0}^{n}{n\choose r}(-1)^{r}2rt^{2r-1}$
and applying this to (\ref{homog 3}) yields
\begin{eqnarray*}
E[C_{n,m}^{0}(Y_{k}^{i})(x)]&=&\partial_{x}^{n}[-2n\vert x\vert^{2}(1-\vert x\vert^{2})^{n-1}Y_{k}^{i}(x)]+(k-n)\partial_x^n\big[(1-|x|^2)^nY_k(x)\big]\big]\\
&=&\partial_{x}^{n}[2n(1-\vert x\vert^{2}-1)(1-\vert x\vert^{2})^{n-1}Y_{k}^{i}(x)]+(k-n)C_{n,m}^{0}(Y_{k}^{i})(x)\\
&=&(n+k)(C_{n,m}^{0}(Y_{k}^{i})(x))-2n\partial_{x}(C_{n,m}^{0}(Y_{k}^{i})(x)).
\end{eqnarray*}
\end{proof}

\begin{Proposition}\label{Proposition3.7}
If 
$f:\mathbb{R}^{m}\to\mathbb{R}_{m}$
has  continuous partial derivatives up to order
$n\geq 0$,
then
\begin{equation}
\partial_{x}^{l}[(1-\vert x\vert^{2})f]=A_{l}\partial_{x}^{l-2}f+B_{l}E\partial_{x}^{l-2}f+C_{l}x\partial_{x}^{l-1}f+(1-\vert x\vert^{2})\partial_{x}^{l}f\label{ABC}
\end{equation}
for $0\leq l\leq n$, where
$$A_{l}=l^{2}+l(m-2)+\frac{(1-m)(1-(-1)^{l})}{2};\ B_{l}=2l-1+(-1)^{l};\ C_{l}=(-1)^{l}-1.$$
\end{Proposition}

\begin{proof}
We first prove (by induction on $m,\ n$) that there exist constants
$A_{l}$, $B_{l}$, $C_l$
such that (\ref{ABC}) holds for $0\leq l\leq n$.
Note that
\eqref{ABC}
is satisfied for 
$l=0$
with
$A_{0}=B_{0}=C_{0}=0$.
By direct calculation, we find 
$$\partial_{x}[(1-\vert x\vert^{2})f(x)]=-2xf(x)+(1-\vert x\vert^{2})\partial_{x}f(x)$$
so that 
\eqref{ABC}
is satisfied for 
$l=1$
with
$A_1 =B_1=0$, $C_1 =-2$.
Suppose there are constants 
$A_k$, $B_{k}$, $C_{k}$
for which 
\eqref{ABC}
holds for 
$2\leq k\leq n-1.$
Then
\begin{align}\label{induction on m,n2}
&\partial_{x}^{k+1}[(1-\vert x\vert^{2})f(x)]\\
&=\partial_{x}\bigg[A_{k}\partial_{x}^{k-2}f(x)+B_{k}E\partial_{x}^{k-2}f(x)+C_{k}x\partial_{x}^{k-1}f(x)+(1-\vert x\vert^{2})\partial_{x}^{k}f(x)\bigg]\nonumber\\
&=A_{k}\partial_{x}^{k-1}f(x)+B_{k}\partial_{x}E\partial_{x}^{k-2}f(x)+C_{k}\partial_{x}Q\partial_{x}^{k-1}f(x)+\partial_{x}[(1-\vert x\vert^{2})\partial_{x}^{k}f(x)]\nonumber\\
&=(A_{k}+B_{k}-mC_{k})\partial_{x}^{k-1}f(x)+(B_{k}-2C_{k})E\partial_{x}^{k-1}f(x)\nonumber\\
&+(-C_{k}-2)x\partial_{x}^{k-1}f(x)+(1-\vert x\vert^{2})\partial_{x}^{k}f(x),
\end{align}
so that
\eqref{ABC}
is verified 
$l=k+1.$
We conclude that there are constants
$A_k$, $B_{k}$, $C_k$
for which 
\eqref{ABC}
is satisfied for 
$0\leq k\leq n-1.$
Comparing
\eqref{induction on m,n2}
with
\eqref{ABC},
we find the recurrence relations
\begin{align}
A_{k+1}&=A_{k}+B_{k}-mC_{k}\label{induction on m,n4}\\
B_{k+1}&=B_{k}-2C_{k}\label{induction on m,n5}\\
C_{k+1}&=-C_{k}-2\label{induction on m,n6}
\end{align}
Equation
\eqref{induction on m,n6} with initial condition $C_0=0$ has solution
$C_{k}=-1+(-1)^{k}$.
Substituting this into (\ref{induction on m,n5})
and applying the initial condition $B_0=0$ gives 
$B_{k}=2k-1+(-1)^{k}$.
Finally, substituting $C_k=-1+(-1)^k$ and $B_{k}=2k-1+(-1)^{k}$ into (\ref{induction on m,n4}) and applying the initial condition $A_0=0$ gives 
$$A_{k}=k^{2}+(m-2)k+\frac{(1-m)(1-(-1)^{k})}{2}.$$
\end{proof}

We now consider recurrence formulae for the Clifford-Legendre and Clifford-Gegenbauer polynomials.
\begin{Th}\label{firstsequenceofgegenbauer}
The Clifford-Legendre polynomials
$\{C_{n,m}^{0}(Y_{k}^{i})(x)\}_{n,k=0}^{\infty}$
satisfy the following recurrence formula;
\begin{equation*}
\partial_{x}[C_{n+1,m}^{0}(Y_{k}^{i})(x)]=\alpha_{n,k,m}[C_{n,m}^{0}(Y_{k}^{i})(x)]\\
+\beta_{n,k}\partial_{x}[C_{n-1,m}^{0}(Y_{k}^{i})(x)]\\
-C_{n+1}x\partial_{x}[C_{n,m}^{0}(Y_{k}^{i})(x)]
\end{equation*}
where
\begin{eqnarray*}
\alpha_{n,k,m}&=&[A_{n+1}+(n+k+1)B_{n+1}-(m+2n+2k)C_{n+1}+C(0,n,m,k)]\\
\beta_{n,k}&=&2n(2C_{n+1}-B_{n+1}),
\end{eqnarray*}
$A_{n}$, $B_{n}$, $C_{n}$
are as in the statement of Proposition \ref{Proposition3.7} and 
$C(0,n,m,k)$
is the given eigenvalue in Theorem \ref{differentialforClifford}.
\end{Th}

\begin{proof}
An application of
Proposition \ref{Proposition3.7} gives
\begin{align}
\partial_{x}[C_{n+1,m}^{0}(Y_{k}^{i})(x)]&=\partial_x\partial_x^{n+1}[(1-|x|^2)(1-|x|^2)^{n}Y_k(x)]\notag\\
&=\partial_x[A_{n+1}\partial_x^{n-1}[(1-|x|^2)^nY_k(x)]+B_{n+1}E\partial_n^{n-1}[(1-|x|^2)^nY_k(x)]\notag\\
&+C_{n+1}x\partial_x^n[(1-|x|^2)^nY_k(x)]+(1-|x|^2)\partial_x^{n+1}[(1-|x|^2)^nY_k(x)]\notag\\
&=A_{n+1}C_{n,m}^{0}(Y_k)(x)+B_{n+1}\partial_xE\partial_x^{n-1}[(1-|x|^2)^nY_k(x)]\notag\\
&+\partial_xQ\partial_x^n[(1-|x|^2)^nY_k(x)]+\partial_x[(1-|x|^2)\partial_xC_{n,m}^{0}(Y_k)(x)]\label{Dirac CL 1}
\end{align}
Remarks
\ref{Remark2.3}
and \ref{Remark2.4}
can be applied to the second and third terms on the right hand side of (\ref{Dirac CL 1}) to obtain
\begin{align}
\partial_{x}[C_{n+1,m}^{0}(Y_{k}^{i})(x)]&=A_{n+1}C_{n,m}^{0}(Y_k)(x)+B_{n+1}(I+E)C_{n,m}^{0}(Y_k)(x)]\notag\\
&+C_{n+1}\partial_xQC_{n,m}^{0}(Y_k)(x)+{\mathcal L}C_{n,m}^{0}(Y_k)(x)\notag\\
&=A_{n+1}C_{n,m}^{0}(Y_k)(x)+B_{n+1}(I+E)C_{n,m}^{0}(Y_k)(x)\notag\\
&+C_{n+1}(-mI-Q\partial_x-2E)C_{n,m}^{0}(Y_k)(x)+{\mathcal L}C_{n,m}^{0}(Y_k)(x)\label{Dirac CL 2}
\end{align}
where 
${\mathcal L}$
is the differential operator of
Theorem \ref{differentialforClifford}. An application of Theorem \ref{differentialforClifford} and
Lemma \ref{EulerofCliffordLegendre} to (\ref{Dirac CL 2}) yields
\begin{align*}
&\partial_{x}[C_{n+1,m}^{0}(Y_{k}^{i})(x)]\\
&=A_{n+1}C_{n,m}^{0}(Y_{k}^{i})(x)+B_{n+1}[C_{n,m}^{0}(Y_{k}^{i})(x)+(n+k)C_{n,m}^{0}(Y_{k}^{i})(x)-2n\partial_xC_{n,m}^{0}(Y_{k}^{i})(x)]\\
&+C_{n+1}[-mC_{n,m}^{0}(Y_{k}^{i})(x)-x\partial_xC_{n,m}^{0}(Y_{k}^{i})(x)-2[(n+k)C_{n,m}^{0}(Y_{k}^{i})(x)-2n\partial_xC_{n,m}^{0}(Y_{k}^{i})(x)]]\\
&+{\mathcal L}C_{n,m}^{0}(Y_{k}^{i})(x)\\
&=[A_{n+1}+(n+k+1)B_{n+1}-(m+2(n+k))C_{n+1}+C(n,m,k)]C_{n,m}^{0}(Y_{k}^{i})(x)\\
&+[4nC_{n+1}-2nB_{n+1}]\partial_xC_{n,m}^{0}(Y_{k}^{i})(x)-C_{n+1}x\partial_xC_{n,m}^{0}(Y_{k}^{i})(x)\\
&=[\alpha_{n,k,m}I+\beta_{n,k}\partial_x-C_{n+1}Q\partial_x]C_{n,m}^{0}(Y_{k}^{i})(x)
\end{align*}
where $\alpha_{n,k,m}$ and $\beta_{n,k}$ are as in the statement of the Theorem.
\end{proof}
The following differential recurrence formula is valid for the Clifford-Gegenbauer polynomials $C_{n,m}^\alpha (Y_k)(x).$
\begin{Th} \label{thm: C-G recurrence}The Clifford-Gegenbauer polynomials $\{C_{n,m}^\alpha (Y_k^i)\}_{n,k=0}^\infty$ satisfy
\begin{align*}
\partial_{x}[C_{n+1,m}^{\alpha}(Y_{k}^{i})(x)]&=4(n+\alpha+1)(n+\alpha+k+\frac{m}{2})\big[C_{n,m}^{\alpha}(Y_{k}^{i})(x)+2\alpha\frac{C_{n-1,m}^{\alpha}(Y_{k}^{i})(x)}{(1-|x|^2)}\big]\\
&-4(n+\alpha+1)(n+\alpha)\partial_{x}[C_{n-1,m}^{\alpha}(Y_{k}^{i})(x)].
\end{align*}
\end{Th}
\begin{proof}
Note  that
\begin{eqnarray*}
\partial_{x}^{2}[(1-\vert x\vert^2)^{n+\alpha+1}Y_{k}^{i}(x)]&=&-2(n+\alpha+1)(1-\vert x\vert^2)^{n+\alpha-1}\hspace*{10cm}\\
&&\times[2(n+\alpha)\vert x\vert^{2}-(2k+m)(1-\vert x\vert^2)]Y_{k}^{i}(x)
\end{eqnarray*}
and consequently
\begin{eqnarray*}
\partial_{x}[C_{n+1,m}^{\alpha}(Y_{k}^{i})(x)]&=&\partial_{x}\big[(1-\vert x\vert^2)^{-\alpha}\partial_{x}^{n-1}[-2(n+\alpha+1)(1-\vert x\vert^2)^{n+\alpha-1}\hspace*{10cm}\\
&&\times[2(n+\alpha)\vert x\vert^{2}-(2k+m)(1-\vert x\vert^2)]Y_{k}^{i}(x)]\big].\\
\end{eqnarray*}
With using 
$\vert x\vert^{2}=-(1-\vert x\vert^2-1),$
\begin{eqnarray*}
\partial_{x}[C_{n+1,m}^{\alpha}(Y_{k}^{i})(x)]&=&2m(n+1+\alpha)\partial_{x}\big[(1-\vert x\vert^2)^{-\alpha}\partial_{x}^{n-1}[(1-\vert x\vert^2)^{n+\alpha}Y_{k}^{i}(x)] \big]\hspace*{10cm}\\
&+&4(1+n+\alpha)\big((n+\alpha)\partial_{x}\big[(1-\vert x\vert^2)^{-\alpha}\partial_{x}^{n-1}[(1-\vert x\vert^2)^{n+\alpha}Y_{k}^{i}(x)]\big]\\
&-&(n+\alpha)\partial_{x}\big[(1-\vert x\vert^2)^{-\alpha}\partial_{x}^{n-1}[(1-\vert x\vert^2)^{(n-1)+\alpha}Y_{k}^{i}(x)]\big]\\
&+&k \partial_{x}\big[(1-\vert x\vert^2)^{-\alpha}\partial_{x}^{n-1}[(1-\vert x\vert^2)^{n+\alpha}Y_{k}^{i}(x)]\big]\big)\\
&=&[2m(n+1+\alpha)+4(n+1+\alpha)(n+\alpha)+4k(n+1+\alpha)]\\
&\times& \big((1-\vert x\vert^2)^{-\alpha}\partial_{x}^{n}[(1-\vert x\vert^2)^{n+\alpha}Y_{k}^{i}(x)]\\
&+&\alpha(2x)(1-\vert x\vert^2)^{-\alpha-1}\partial_{x}^{n-1}[(1-\vert x\vert^2)^{n+\alpha}Y_{k}^{i}(x)]\big)\\
&-&4(n+1+\alpha)(n+\alpha)\partial_{x}\big[(1-\vert x\vert^2)^{-\alpha}\partial_{x}^{n-1}[(1-\vert x\vert^2)^{(n-1)+\alpha}Y_{k}^{i}(x)]\big].\\
\end{eqnarray*}
By the 
definition \ref{Definition of Gegenbauer}, we have that
\begin{eqnarray*}
\partial_{x}[C_{n+1,m}^{\alpha}(Y_{k}^{i})(x)]&=&[2m(n+1+\alpha)+4(n+1+\alpha)(n+\alpha)+4k(n+1+\alpha)]\big(C_{n,m}^{\alpha}(Y_{k}^{i})(x)\\
&+&2\alpha x(1-\vert x\vert^2)^{-\alpha-1}\partial_{x}^{n-1}[(1-\vert x\vert^2)^{(n-1)+\alpha}Y_{k}^{i}(x)]\big)\\
&-&4(n+1+\alpha)(n+\alpha)\partial_{x}C_{n-1,m}^{\alpha}(Y_{k}^{i})(x)\\
&=&4(n+1+\alpha)(n+\alpha+k+\frac{m}{2})\big[C_{n,m}^{\alpha}(Y_{k}^{i})(x)\\
&+&2\alpha(1-\vert x\vert^2)^{-1}C_{n-1,m}^{\alpha}(Y_{k}^{i})(x)\big]\\
&-&4(n+1+\alpha)(n+\alpha)\partial_{x}[C_{n-1,m}^{\alpha}(Y_{k}^{i})(x)]
\end{eqnarray*}
which is the desired recurrence.
\end{proof}	
Putting $\alpha =0$ in Theorem \ref{thm: C-G recurrence} gives the following  recurrence relation for Clifford-Legendre polynomials.
\begin{Corollary} \label{corollarysecondsequence}
	For $n\geq 1$, $k\geq 0$ and $1\leq i\leq d_k$, the Clifford-Legendre polynomials $C_{n,m}^0(Y_k^i)(x)$ satisfy the recurrence relation
$$\partial_{x}C_{n+1,m}^{0}(Y_{k}^{i})(x)=4(n+1)[(n+k+\frac{m}{2})C_{n,m}^{0}(Y_{k}^{i})(x)-n\partial_{x}C_{n-1,m}^{0}(Y_{k}^{i})(x)].$$
\end{Corollary}
\begin{Remark}
When
$n$
is odd, the recurrence formulas of
Theorem
\ref{firstsequenceofgegenbauer}
and 
Corollary \ref{corollarysecondsequence}
for the Clifford-Legendre polynomials are identical. Since $B_{n}=2n-1+(-1)^{n},$ if $n$ is odd, $B_{n+1}=2(n+1).$ So the coefficients 
$\beta_{n,k}$ from Proposition \ref{firstsequenceofgegenbauer} become $\beta_{n,k}=-4n(n+1)$ and the coefficients $\alpha_{n,k}$ become
$\alpha_{n,k}=4(n+1)(n+k+\frac{m}{2})$.
\end{Remark}
The next result provides an  explicit representation for the Clifford-Legendre polynomials for the even and odd cases separately.
\begin{Th}\label{thm: representationofCl-Legendre}
Let $N,k\geq 0$ and $1\leq i\leq d_k$. Then we have
\begin{eqnarray*}
&&C_{2N+1,m}^{0}(Y_{k}^{i})(x)=-\frac{2^{2N+1}(2N+1)!}{N!}\sum\limits_{l=0}^{N}{N\choose l}\frac{\Gamma(l+k+\frac{m}{2}+N+1)}{\Gamma(l+k+\frac{m}{2}+1)}(-1)^{l}\vert x\vert^{2l}xY_{k}^{i}(x),\\	
&&C_{2N,m}^{0}(Y_{k}^{i})(x)=\frac{2^{2N}(2N)!}{N!}\sum\limits_{l=0}^{N}{N\choose l}\frac{\Gamma(l+k+\frac{m}{2}+N)}{\Gamma(l+k+\frac{m}{2})}(-1)^{l}\vert x\vert^{2l}Y_{k}^{i}(x).
\end{eqnarray*}
\end{Th}
\begin{proof}
When $n=2N+1$ is odd, repeated application of 
Lemma \ref{lem: D and Delta on Y_k} and the binomial theorem give 
\begin{eqnarray*}
	&&C_{2N+1,m}^{0}(Y_{k}^{i})(x)\\
	&=&\partial_{x}^{2N+1}\big[(1-\vert x\vert^{2})^{2N+1}Y_{k}^{i}(x)\big] \hspace*{10cm}\\
	&=&\sum\limits_{j=0}^{2N+1}{2N+1\choose j}(-1)^{j} \partial_{x}^{2N+1}\vert x\vert^{2j}Y_{k}^{i}(x)\\
	&=&-\sum\limits_{j=N}^{2N+1}{2N+1\choose j}(-1)^{j+N}2^{N}\frac{j!}{(j-(N+1))!}\frac{\Gamma(j+k+\frac{m}{2}-1)}{\Gamma(j+k+\frac{m}{2}-(N+1))}\vert x\vert^{2j-2(N+1)}Y_{k}^{i}(x)\\
	&=&-\sum\limits_{l=0}^{N}{2N+1\choose N+l}(-1)^{l+1}2^{2N}\frac{(N+1+l)!}{(l+1)!}\frac{\Gamma(\frac{m}{2}+k+l+N+1)}{\Gamma(\frac{m}{2}+k+l+1)}2(l+1)x\vert x\vert^{2l}Y_{k}^{i}(x)\\
	&=&-\frac{2^{2N+1}(2N+1)!}{N!}\sum\limits_{l=0}^{N}{N\choose l}\frac{\Gamma(l+k+\frac{m}{2}+N+1)}{\Gamma(l+k+\frac{m}{2}+1)}(-1)^{l}\vert x\vert^{2l}xY_{k}^{i}(x).
\end{eqnarray*}
Similarly, when $n=2N$ is even,
\begin{eqnarray*}
C_{2N,m}^{0}(Y_{k}^{i})(x)&=&\partial_{x}^{2N}\big[(1-\vert x\vert^{2})^{2N}Y_{k}^{i}(x)\big] \hspace*{10cm}\\
	&=&\sum\limits_{j=0}^{2N}{2N\choose j}(-1)^{j} \partial_{x}^{2N}\vert x\vert^{2j}Y_{k}^{i}(x)\\
	&=&\sum\limits_{j=N}^{2N}{2N\choose j}(-1)^{j+N}2^{N}\frac{(j)!}{(j-N)!}\frac{\Gamma(j+k+\frac{m}{2})}{\Gamma(j+k+\frac{m}{2}-N)}\vert x\vert^{2j-2N}Y_{k}^{i}(x)\\
	&=&\sum\limits_{l=0}^{N}{2N\choose N+l}(-1)^{l}2^{2N}\frac{(N+l)!}{(l)!}\frac{\Gamma(\frac{m}{2}+k+l+N)}{\Gamma(\frac{m}{2}+k+l)}\vert x\vert^{2l}Y_{k}^{i}(x)\\
	&=&\frac{2^{2N}(2N)!}{N!}\sum\limits_{l=0}^{N}{N\choose l}(-1)^{l}\frac{\Gamma(\frac{m}{2}+k+l+N)}{\Gamma(\frac{m}{2}+k+l)}\vert x\vert^{2l}Y_{k}^{i}(x).\\
\end{eqnarray*}
\end{proof}
\begin{Corollary}
If
$C_{2N,m}^{0}(Y_{k}^{j})(x),$
and
$C_{2N+1,m}^{0}(Y_{k}^{j})(x)$
are Clifford-Legendre polynomials, then there exist polynomials
$P_{N,k},$
and
$Q_{N,k}$
of degree
$m$
such that
\begin{eqnarray*}
C_{2N,m}^{0}(Y_{k}^{j})(x)&=&P_{N,k,m}(\vert x\vert^{2})Y_{k}^{j}(x),\\
C_{2N+1,m}^{0}(Y_{k}^{j})(x)&=&Q_{N,k,m}(\vert x\vert^{2})xY_{k}^{j}(x).
\end{eqnarray*}
\end{Corollary}	
\bigskip

\section{Normalisation of the  Clifford-Legendre Polynomials}
In this section we compute the Fourier transforms of the Clifford-Legendre and their
$L^{2}$-norms as a consequence. Plots of the polynomials are provided, and a curious degeneracy observed in the case $m=2$.

We consider the Clifford algebra-valued inner product of the functions 
$f,g:{\mathbb R}^m\to{\mathbb R}_m$
by
$$\langle f,g\rangle=\int\limits_{\mathbb{R}^{m}}\overline{f(x)}g(x)\, dx,$$
where
$dx$ is Lebesgue measure on 
$\mathbb{R}^{m}$. The associated norm $\|\cdot\|_2$ is given by
$$\Vert f\Vert_{2}^{2}=[\langle f,f\rangle]_{0}=\left(\int_{{\mathbb R}^m}|f(x)|^2\, dx\right)^{1/2}.$$
The right Clifford-module of Clifford algebra-valued measurable functions on
$\mathbb{R}^{m}$
for which 
$\Vert f\Vert_{2}<\infty$
is a right Hilbert Clifford-module which we denote by 
$L^{2}(\mathbb{R}^{m},{\mathbb R}_m)$.

The standard tensorial multi-dimensional Fourier transform given by:
\begin{equation}\label{equation2.20}
\mathcal{F}f(\xi)=\int\limits_{\mathbb{R}^{m}}\exp(-2\pi i\langle x,\xi\rangle)f(x)\, dx
\end{equation}
whenever $f\in L^1({\mathbb R}^m,{\mathbb R}_m)$. As is shown in \cite{delanghe2012clifford}, the Fourier transform extends to a unitary mapping on $L^2({\mathbb R}^m,{\mathbb R}_m)$.

\begin{Th}(\textbf{Plancherel theorem})\label{Plancherel theorem}
For all
	$f,g\in L^2(\mathbb{R}^{m},{\mathbb R}_m)$
	the Parseval formula holds:
	$$\langle f,g\rangle=\langle \mathcal{F}f,\mathcal{F}g\rangle.$$
	In particular, for each 
	$f\in L^2(\mathbb{R}^{m},{\mathbb R}_m)$
	one has:
	$$\Vert f\Vert_{2}=\Vert \mathcal{F}f\Vert_{2}.$$
\end{Th}

By Theorem \ref{Clifford-Stokes theorem}, it is possible to prove the following orthogonality property of homogeneous monogenic polynomials.

\begin{Lemma}\label{fromStokes}	
Let $Y_{k}\in M_{l}^{+}(k)$
and
$Y_{k'}\in M_{l}^{+}(k')$. Then
$$\int\limits_{S^{m-1}}\overline{Y_{k}(\theta)}\theta Y_{k'}(\theta)d\theta=0.$$
\end{Lemma}

The orthogonality of the Clifford-Legendre polynomials is proved in \cite{delanghe2012clifford}.

\begin{Lemma}\label{orthogonalitybasis}
The Clifford-Legendre polynomials 
$$\{C_{n,m}^{0}(Y_{k}^{i})(x):\ n\geq 0,\ k\geq 0,\ 1\leq i\leq d_k\}$$
form an orthogonal basis for the functions in
$L^{2}(B(1),{\mathbb R}_m)$.
\end{Lemma}

The following well-known result appears as Lemma 9.10.2 in \cite{andrews1999special}.
\begin{Lemma}\label{Lemma 4.3}	
Let
$\hat{\xi},\theta\in S^{m-1}$, $r>0$ and $Y_k\in M_l^+(k)$. Then 
$$\int\limits_{S^{m-1}}e^{-2\pi ir\langle \hat{\xi},\theta \rangle}Y_{k}(\theta)d\sigma(\theta)=\frac{2\pi(-i)^{k}}{r^{\frac{m}{2}-1}}J_{k+\frac{m}{2}-1}(2\pi r)Y_{k}(\hat{\xi}),$$
where
$J_{k+\frac{m}{2}-1}$
is a Bessel function of the first kind.
\end{Lemma}

\begin{Lemma}\label{Lemma4.4}
If 
$f\in C^n(B(1),{\mathbb R}^m)$ $(n\geq 1)$ and $0\leq k\leq n$, then  
\begin{equation}\label{equation4.2}
\partial_{x}^{k}((1-\vert x\vert^{2})^{n}f(x))=(1-\vert x\vert^{2})^{n-k}f_{k}(x)
\end{equation}
with
$f_{k}\in C^{n-k}(B(1),{\mathbb R}_m)$.
\end{Lemma}
\begin{proof}
The proof is by induction on
$k$.
Equation
\eqref{equation4.2}
is clearly true when 
$k=0.$
Suppose
\eqref{equation4.2}
holds for 
$k=l\; (0\leq l\leq n-1),$
i.e., 
$$\partial_{x}^{l}((1-\vert x\vert^{2})^{n}f(x))=(1-\vert x\vert^{2})^{n-l}f_{l}(x)$$
with 
$f_{l}\in C^{n-l}(B(1),{\mathbb R}^m)$.
Then,
\begin{eqnarray*}
\partial_{x}^{l+1}((1-\vert x\vert^{2})f(x))&=&\partial_{x}[(1-\vert x\vert^{2})^{n-l}f_{l}(x)]\\
&=&\sum\limits_{j=1}^{m}e_{j}[(n-l)(1-\vert x\vert^{2})^{n-l-1}(-2x_{j})f_{l}(x)+(1-\vert x\vert^{2})^{n-l}\partial_{x_{j}}f_{l}(x)]\\
&=&-2(n-l)x(1-\vert x\vert^{2})^{n-l-1}f_{l}(x)+(1-\vert x\vert^{2})^{n-l}\partial_{x}f_{l}(x)\\
&=&(1-\vert x\vert^{2})^{n-l-1}\partial_{x}f_{l+1}(x),
\end{eqnarray*}
where
$f_{l+1}(x)=-2(n-l)xf_{l}(x)+(1-|x|^2)\partial_{x}f_{l}(x).$
\end{proof}
	
\begin{Th}
The Fourier transform of the restriction of the Clifford-Legendre polynomial
$C_{n,m}^{0}(Y_{k}^{i})(x)$
to the unit ball
$B(1),$
is given by
\begin{equation}
\mathcal{F}(C_{n,m}^{0}(Y_{k}^{i}))(\xi)=(-1)^{k}i^{n+k}2^{n}n!\xi^n\frac{J_{k+\frac{m}{2}+n}(2\pi \vert\xi\vert)}{\vert \xi\vert^{\frac{m}{2}+n+k}}Y_{k}^{i}(\xi).
\end{equation}
\end{Th}
\begin{proof}
We apply the Rodrigues' formula for the Clifford-Legendre polynomials and the Clifford-Stokes theorem to find 
\begin{eqnarray*}
\mathcal{F}(C_{n,m}^{0}(Y_{k}^{i}))(\xi)&=&\int\limits_{B(1)}e^{-2\pi i\langle x,\xi\rangle}\partial_{x}^{n}[(1-\vert x\vert^{2})^{n}Y_{k}^{i}(x)]\, dx\hspace*{10cm}\\
&=&\int\limits_{S^{m-1}}e^{-2\pi i \langle x,\xi\rangle}x\,\partial_{x}^{n-1}[(1-\vert x\vert^{2})^{n}Y_{k}^{i}(x)]\, dx\\
&-&\int\limits_{B(1)}(e^{-2\pi i \langle x,\xi\rangle}\partial_{x})(\partial_{x}^{n-1}[(1-\vert x\vert^{2})^{n}Y_{k}^{i}(x)])\, dx .
\end{eqnarray*} 
By Lemma \ref{Lemma4.4}, the restriction of
$\partial_{x}^{n-1}[(1-\vert x\vert^{2})^{n}Y_{k}^{i}(x)]$
to the unit sphere
$S^{m-1}$
is zero, so that
$$\mathcal{F}(C_{n,m}^{0}(Y_{k}^{i}))(\xi)=(2\pi i\xi)\int\limits_{B(1)}e^{-2\pi i \langle x,\xi\rangle}(\partial_{x}^{n-1}[(1-\vert x\vert^{2})^{n}Y_{k}^{i}(x)])\, dx.$$
By applying Theorem \ref{Clifford-Stokes theorem} repeatedly, we find
\begin{eqnarray*}
\mathcal{F}(C_{n,m}^{0}(Y_{k}^{i}))(\xi)&=&(2\pi i\xi)^{n}\int\limits_{B(1)}e^{-2\pi i \langle x,\xi\rangle}[(1-\vert x\vert^{2})^{n}Y_{k}^{i}(x)]dx\hspace*{10cm}\\
&=&(2\pi i\xi)^{n}\int\limits_{0}^{1}r^{m-1+k}(1-r^{2})^{n}\int\limits_{S^{m-1}}e^{-2\pi i r\langle \omega,\xi\rangle}Y_{k}^{i}(\omega)d\omega\, dr\\
&=&(2\pi i\xi)^{n}(2\pi)(-i)^{k}\frac{Y_{k}^{i}(\frac{\xi}{\vert \xi\vert})}{\vert \xi\vert^{\frac{m}{2}-1}}\int\limits_{0}^{1}r^{\frac{m}{2}+k}(1-r^{2})^{n}J_{k+\frac{m}{2}-1}(2\pi r\vert \xi\vert)dr,
\end{eqnarray*}
where we have used the homogeneity of $Y_k^i$ and Lemma \ref{Lemma 4.3} in the last step.
 The last integral can be computed from
(\cite{gradshteyn2007ryzhik}, 6.567 \;1)
to yield the result.
\end{proof}	

\begin{Corollary}\label{NormofCl-Legendre}
The 
$L^{2}-$norm of the restriction of the Clifford-Legendre polynomial
$C_{n,m}^{0}(Y_{k}^{i})$
to the unit ball $B(1)$ is given by
$$\Vert C_{n,m}^{0}(Y_{k}^{i})(x)\Vert^{2}_{2}=\dfrac{2^{2n}(n!)^{2}}{2k+2n+m}.$$
\end{Corollary}
\begin{proof}
We apply the 
Theorem
\ref{Plancherel theorem} 
and the assumption that
$Y_{k}^{i}$
is 
$L^{2}(S^{m-1})$-normalized to find 
\begin{eqnarray*}
\Vert C_{n,m}^{0}(Y_{k}^{i})(x)\Vert^{2}&=&\Vert \mathcal{F}(C_{n,m}^{0}(Y_{k}^{i}))\Vert^{2}\hspace*{10cm}\\
&=&2^{2n}(n!)^{2}\int\limits_{\mathbb{R}^{m}}\bigg\vert\dfrac{Y_{k}^{i}(\xi) J_{k+\frac{m}{2}+n}(2\pi \vert \xi\vert)}{\vert \xi\vert^{\frac{m}{2}+k}}\bigg\vert^{2}d\xi\\
&=&2^{2n}(n!)^{2}\int\limits_{0}^{\infty}\bigg(\int\limits_{S^{m-1}}\vert Y_{k}^{i}(\omega)\vert^{2}d\omega\bigg)\vert J_{k+\frac{m}{2}+n}(2\pi r)\vert^{2}r^{-1}dr\\
&=&\dfrac{2^{2n}(n!)^{2}}{2k+2n+m}
\end{eqnarray*}
where the last integral has computed from (\cite{gradshteyn2007ryzhik}, 6.5742).
\end{proof}
 We therefore define the normalised Clifford-Legendre polynomials $\barC_{n,m}^0(Y_k^1)$ by 
 \begin{equation}
 \barC_{n,m}^0(Y_k^i)=\frac{\sqrt{2k+2n+m}}{2^nn!}C_{n,m}^0(Y_k).\label{normalized C_L}
 \end{equation}
According to the Definition
\ref{left monogenic homogeneous polynomial}, 
in dimension
$m=2$ we have
$$\dim M_{l}^{+}(k)=\frac{(m+k-2)!}{(m-2)!k!}=1.$$
Consequently, when $m=2$ the function 
\begin{equation}
Y_{k}(r\cos\theta,r\sin\theta)=\dfrac{r^{k}}{\sqrt{2\pi}}[e_1\cos k\theta -e_2\sin k\theta ]
\label{2d Y_k}
\end{equation}
 itself forms an orthonormal basis for $M_l^+(k)$ and in this case the Clifford-Legendre polynomials (described explicitly in Theorem \ref{thm: representationofCl-Legendre}) take the form 
\begin{align*}
C_{2N,2}^0(Y_k)(x)&=F_{N,k}^1e_1+F^2_{N,k}(x)e_2\\
C_{2N+1,2}^0(Y_k)(x)&=G_{N,k}^1+G_{N,k}^2(x)e_{12}
\end{align*}
where $F_{N,k}^1$, $F_{N,k}^2$, $G_{N,k}^1$, $G_{N,k}^2$ are real-valued functions defined on the unit ball $B(1)$. In Figures \ref{Figure41}-\ref{Figure44} below, these functions are plotted for various values of $N$ and $k$.
\begin{figure}[h]\label{Figure41}
	\centering
	\includegraphics[width=0.4\linewidth]{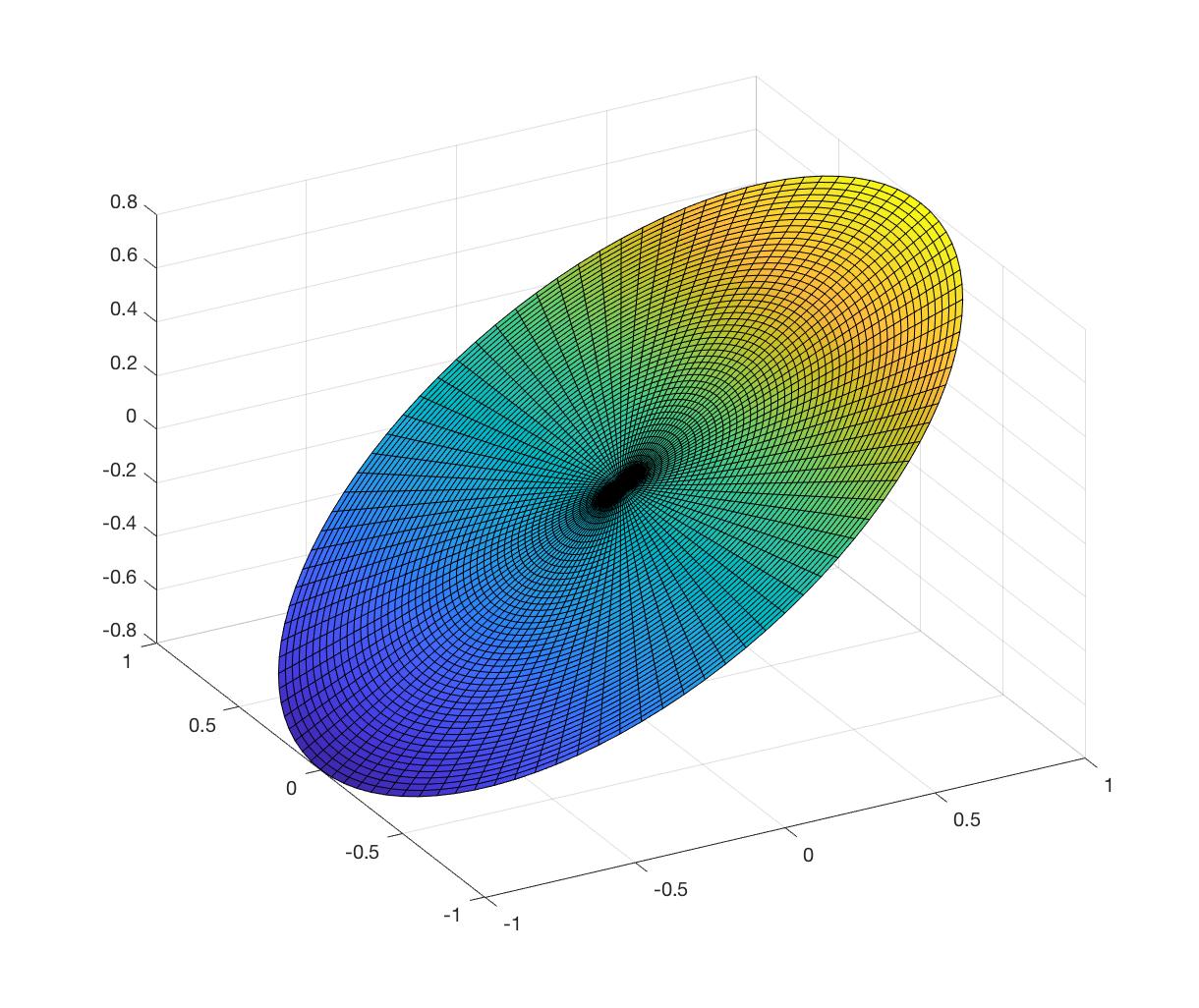}
	\caption{{\small Graph of $e_1$ part of normalized Clifford-Legendre polynomial $C_{0,2}^0(Y_1)$.}}
\end{figure}
\begin{figure}[h]\label{Figure42}
	\centering
	\includegraphics[width=0.4\linewidth]{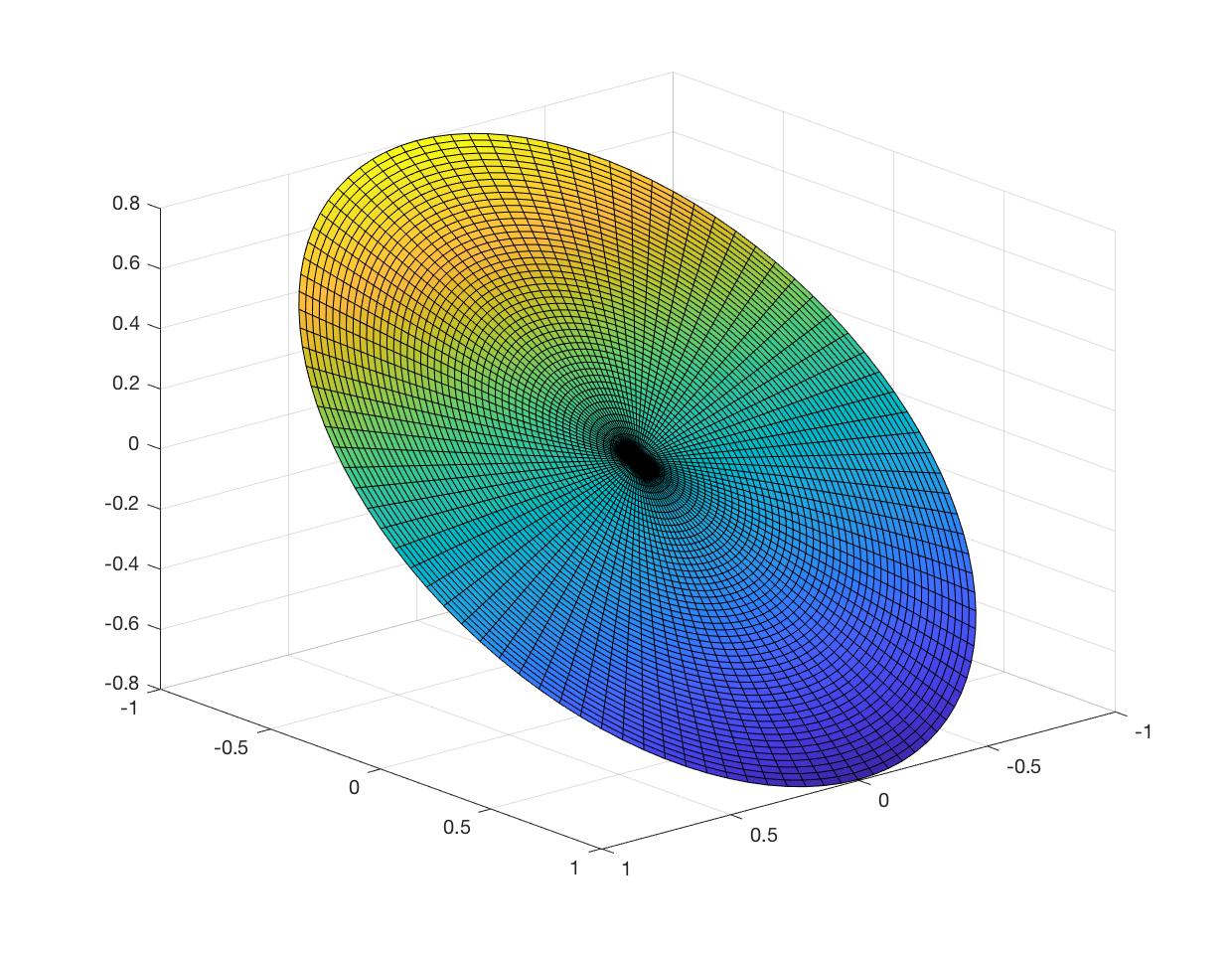}
	\caption{{\small Graph of $e_2$ part of normalized Clifford-Legendre polynomial $C_{0,2}^0(Y_1)$.}}
\end{figure}
\begin{figure}[h]\label{Figure43}
	\centering
	\includegraphics[width=0.4\linewidth]{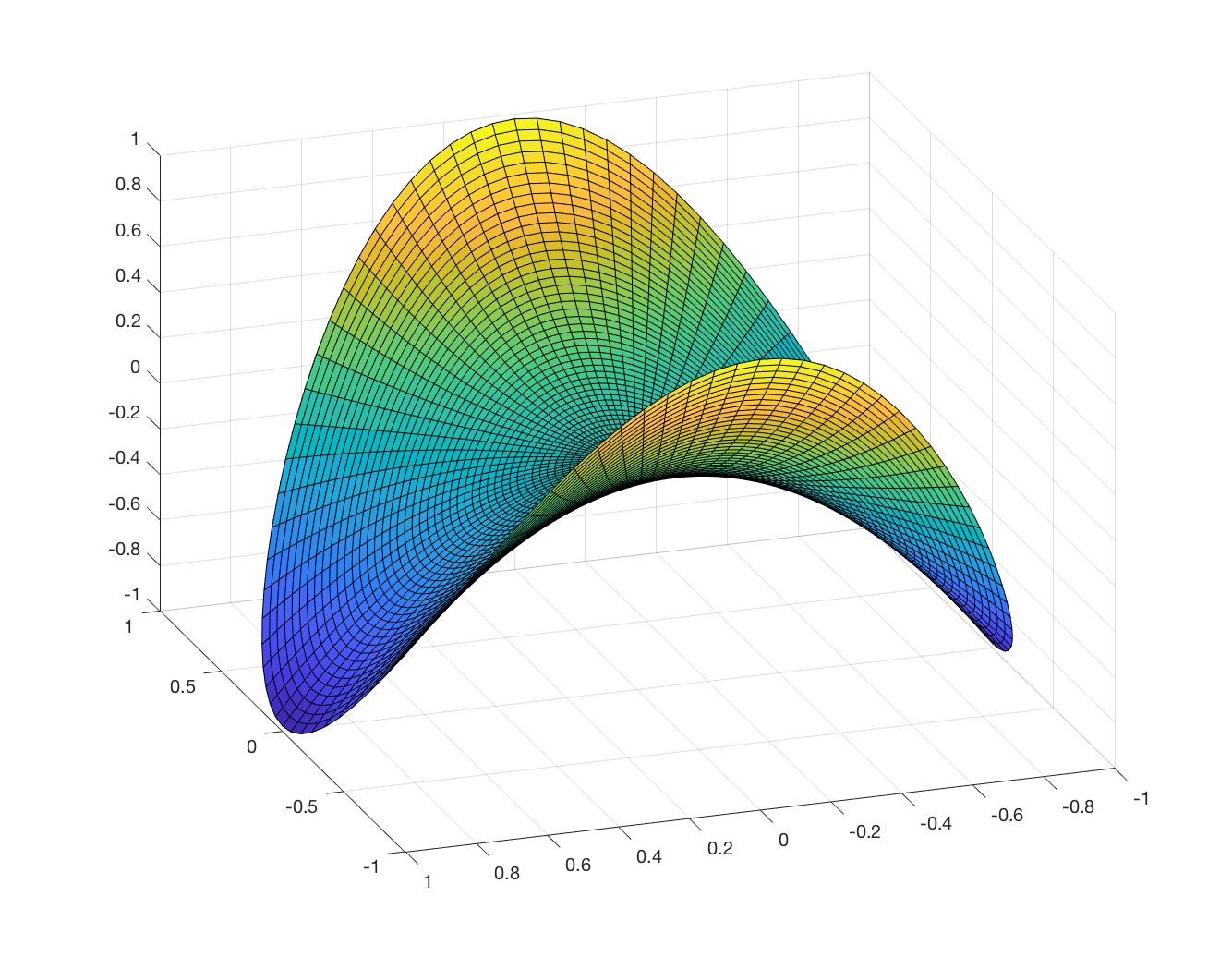}
	\caption{{\small Graph of real part of normalized Clifford-Legendre polynomial $C_{1,2}^0(Y_1)$.}}
\end{figure}
\begin{figure}[h]\label{Figure44}
	\centering
	\includegraphics[width=0.4\linewidth]{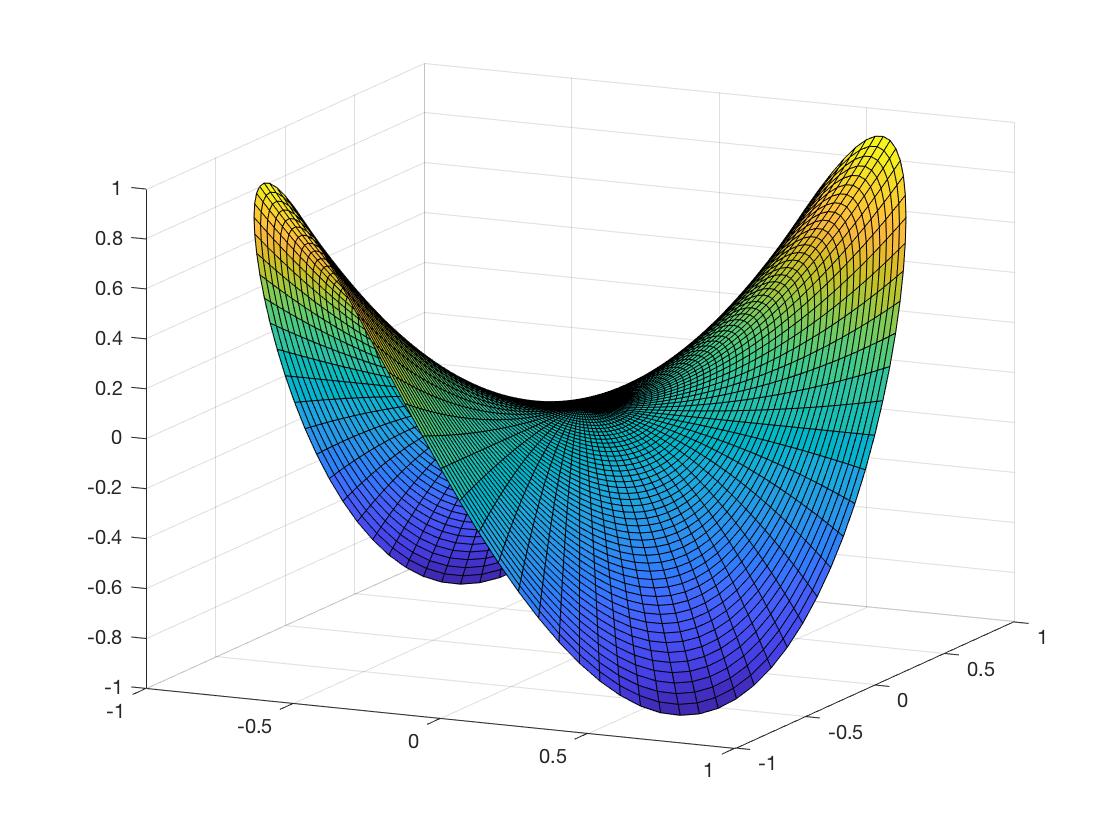}
	\caption{{\small Graph of $e_{12}$ part of normalized Clifford-Legendre polynomial $C_{1,2}^0(Y_1)$.}}
\end{figure}
In dimension $m=2$, it can be easily seen from (\ref{2d Y_k}) that 
$xY_{k}(x)=e_{1}Y_{k+1}(x)$.
As a consequence, we have the following degeneracy between even and odd Clifford-Legendre polynomials.epeated
\begin{Th}\label{Even and Odd Clifford Legendre}
In dimension $m=2$, the normalised Clifford-Legendre polynomials satisfy
$$\barC_{2N+1,2}^{0}(Y_{k})(x)=-e_{1}\barC_{2N,2}^{0}(Y_{k+1})(x).$$
\end{Th}

\begin{proof}
Putting $m=2$ in the explicit representations of Theorem \ref{thm: representationofCl-Legendre} and applying the normalization (\ref{normalized C_L}) gives	
$$\barC_{2N+1,2}^{0}(Y_{k})(x)=-\frac{\sqrt{4N+2k+4}}{N!}\sum_{l=0}^{N}{N\choose l}{l+k+N+1\choose N}(-1)^{l}\vert x\vert^{2l}xY_{k}(x).$$
However, $xY_k=e_1Y_{k+1}$, so  
\begin{align*}
\barC_{2N+1,2}^0(Y_k)(x)&=-e_1\frac{\sqrt{4N+2k+4}}{N!}\sum_{l=0}^{N}{N\choose l}{l+k+N+1\choose N}(-1)^{l}\vert x\vert^{2l}Y_{k+1}(x)\\
&=-e_1\barC_{2N,2}^0(Y_{k+1})(x).
\end{align*}
\end{proof}
\section{Connections Between Clifford Legendre polynomials and Jacobi Polynomials}
In this section, we prove that the radial part of the Clifford-Legendre polynomials are shifted re-scaled Jacobi polynomials. This observation provides an explanation of the observed interlacing of the zeros of the Clifford-Legendre polynomials.

Let $m\geq 2$ be arbitrary and 
$n=2N$
be even. We can write 
\begin{equation}
C^{0}_{2N,m}(Y_{k})(x)=P_{N,k,m}(\vert x\vert^{2})Y_{k}(x)\label{radial decomp}
\end{equation}
 and we call the polynomials $P_{N,k,m}(|x|^2)$ the radial part of the Clifford-Legendre polynomial $C_{2N,m}^0(Y_k)$. Let ${\mathcal L}=-(\Delta +2x\partial_x)$ be the differential operator that appears in Theorem \ref{differentialforClifford}. Then we have
$${\mathcal L}(C_{2N,m}^{0}(Y_{k}^{i})(x))=C(0,2N,m,k)C_{2N,m}^{0}(Y_{k}^{i})(x).$$
We aim to determine a differential operator $T_0$ for which $P_{N,k,m}$ is an eigenfunction.
 We have
\begin{equation}
{\mathcal L}(C_{2N,m}^0(Y_k^i)(x)=[\partial_{x}(1-\vert x\vert^{2})\ ]C^{0}_{2N,m}(Y_{k}^{i})(x)=C(0,2N,m,k)C_{2N,m}^{0}(Y_{k}^{i})(x).\label{C-L eigenfunction}
\end{equation}
On the other hand, since $Y_k^i$ is left monogenic,
\begin{eqnarray*}
{\mathcal L}(C_{2N,m}^0(Y_k^i))(x)	&=&\partial_{x}\big[(1-\vert x\vert^{2})\partial_{x}(P_{N,k,m}(\vert x\vert^{2})Y_{k}^{i}(x))\big]\\
	&=&\partial_{x}\bigg[(1-\vert x\vert^{2})\sum\limits_{j=1}^{2}e_{j}[P'_{N,k,m}(\vert x\vert^{2})2x_{j}Y_{k}^{i}(x)+P_{N,k,m}(\vert x\vert^{2})\frac{\partial}{\partial x_{j}}Y_{k}^{i}(x)]\bigg]\\
	&=&\partial_{x}\big[(1-\vert x\vert^{2})2xP'_{N,k,m}(\vert x\vert^{2})Y_{k}^{i}(x)\big]\\
	&=&2\sum\limits_{j=1}^{2}e_{j}\bigg[-2x_{j}xP'_{N,k,m}(\vert x\vert^{2})Y_{k}^{i}(x)+(1-\vert x\vert^{2})e_{j}P'_{N,k,m}(\vert x\vert^{2})Y_{k}^{i}(x)\\
	&+&(1-\vert x\vert^{2})xP''_{N,k,m}(\vert x\vert^{2})2x_{j}Y_{k}^{i}(x)+(1-\vert x\vert^{2})xP'_{N,k,m}(\vert x\vert^{2})\frac{\partial Y_{k}^{i}}{\partial x_{j}}\bigg].
	\end{eqnarray*}
	However, since $e_jx=-xe_j-2x_j$ and $EY_k^i=kY_k^i$, we have
	\begin{eqnarray*}
{\mathcal L}(C_{2N,m}^0(Y_k^i))(x)	&=&2\bigg[-2x^{2}P'_{N,k,m}(\vert x\vert^{2})Y_{k}^{i}(x)-m(1-\vert x\vert^{2})P'_{N,k,m}(\vert x\vert^{2})Y_{k}^{i}(x)\\
	&+& 2(1-\vert x\vert^{2})x^{2}P''_{N,k,m}(\vert x\vert^{2})Y_{k}^{i}(x)+(1-\vert x\vert^{2})P'_{N,k,m}(\vert x\vert^{2})(-xe_{j}-2x_{j})\frac{\partial Y_{k}^{i}}{\partial x_{j}}\bigg]\\
	&=&2\bigg[2\vert x\vert^{2} P'_{N,k,m}(\vert x\vert^{2})Y_{k}^{i}(x)-m(1-\vert x\vert^{2})P'_{N,k,m}(\vert x\vert^{2})Y_{k}^{i}(x)\\
	&-& 2(1-\vert x\vert^{2})\vert x\vert^{2} P''_{N,k,m}(\vert x\vert^{2})Y_{k}^{i}(x)-2k(1-\vert x\vert^{2})P'_{N,k,m}(\vert x\vert^{2})Y_{k}^{i}(x)\bigg]\\
	&=&2\bigg[-2|x|^2(1-|x|^2)P_{N,k,m}''(|x|^2)\\
	&+&(-m-2k+(2+m+2k)|x|^2)P_{N,m,k}'(|x|^2)\bigg]Y_k^i(x).\\
\end{eqnarray*}
We conclude from  (\ref{C-L eigenfunction}) that the functions $P_{N,k,m}$ satisfy
\begin{equation}
t(1-t)P_{N,k,m}''(t)+\bigg[\bigg(\frac{m}{2}+k\bigg)-\bigg(1+\frac{m}{2}+k\bigg)t\bigg]P_{N,k,m}'(t)=-\frac{C(0,2N,m,k)}{4}P_{N,k,m}(t),\label{radial diffeq}
\end{equation}
i.e., $T_0P_{N,k,m}=-\dfrac{C(0,2N,m,k)}{4}P_{N,k,m}$ where $T_0$ is the differential operator 
$$T_0=t(1-t)\frac{d^2}{dt^2}+\bigg[\bigg(\frac{m}{2}+k\bigg)-\bigg(1+\frac{m}{2}+k\bigg)t\bigg]\frac{d}{dt}$$
where $t\in (0,1)$.
On putting $t=\dfrac{s+1}{2}$ $(s\in (-1,1)$ and $R_{N,k,m}(s)=P_{N,k,m}(t)$, equation (\ref{radial diffeq}) becomes
$$S_0R_{N,k,m}=-\frac{C(0,2N,m,k)}{4}R_{N,k,m}$$
where $S_0$ is the differential operator
\begin{equation}
S_0=(1-s^2)\frac{d^2}{ds^2}+\bigg[\bigg(k+\frac{m}{2}-1\bigg)-\bigg(\frac{m}{2}+k+1\bigg)s\bigg]\frac{d}{ds}.\label{S_0 def}
\end{equation}
The Jacobi polynomials   
$P_{n}^{(\alpha,\beta)}(x)$ $(n\geq 0,\ \alpha ,\beta >-1,\ x\in [-1,1]$ satisfy the differential equation
\begin{equation}\label{jacobiequation}
(1-x^{2})y''(x)+[\beta-\alpha-(\alpha+\beta+2)x]y'(x)=-n(n+\alpha +\beta +1) y(x).
\end{equation}
Since the eigenvalues of the Jacobi differential equation operator are non-degenerate we have
\begin{equation}
P_{N,k,m}(\vert x\vert^{2})=c_{N,k,m}P_{N}^{(0,k+\frac{m}{2}-1)}(2\vert x\vert^{2}-1)
\end{equation}
for real constants $c_{N,k,m}$.
We aim to find 
$c_{N,k,m}$
when
$C_{2N,m}^{0}(Y_{k}^{i})(x)$
is normalized. In this case, the homogeneity of $Y_k^i$ and its normalization on the unit sphere give
\begin{eqnarray}\label{findingcNk1}
1&=&\int\limits_{B(1)}\vert \barC_{2N,m}^{0}(Y_{k}^{i})(x) \vert ^{2}dx=\int\limits_{B(1)}\vert P_{N,k,m}(\vert x\vert^{2})Y_{k}^{i}(x) \vert^{2}\, dx\hspace*{2cm}\nonumber\\
&=&\int\limits_{B(1)}\vert c_{N,k,m} P_{N}^{(0,k+\frac{m}{2}-1)}(2\vert x\vert^{2}-1)Y_{k}^{i}(x)\vert^{2}\, dx\nonumber\\
&=&c_{N,k,m}^{2}\int\limits_{S^{m-1}}\int\limits_{0}^{1}\vert P_{N}^{(0,k+\frac{m}{2}-1)}(2r^{2}-1)\vert^{2} r^{2k}\vert Y_{k}^{i}(\theta)\vert^{2} r^{m-1} \, dr \, d\theta\nonumber\\
&=&\frac{c_{N,k,m}^{2}}{2^{k+\frac{m}{2}+1}}\int\limits_{-1}^{1}(s+1)^{k+\frac{m}{2}-1}\vert P_{N}^{(0,k+\frac{m}{2}-1)}(s)\vert^{2}\, ds.
\end{eqnarray}
But from \cite{gradshteyn2007ryzhik} page $983$ we see that the final integral in (\ref{findingcNk1}) equals
$\frac{2^{k+\frac{m}{2}}}{k+\frac{m}{2}+2N}$ and we conclude that 
$c_{N,k,m}=\pm\sqrt{2(k+\frac{m}{2}+2N)}$. The odd case may be treated similarly. These calculations are summarised below.
\begin{Th}\label{thm: CL and Jacobi}
Let 
$\barC_{2N,m}^{0}(Y_{k}^{i})(x)$
and
$\barC_{2N+1,m}^{0}(Y_{k}^{i})(x)$
be normalized Clifford-Legendre polynomials. Then the radial part of these functions are shifted, scaled and renormalised Jacobi polynomials, i.e., 
\begin{eqnarray*}
&&\barC_{2N,m}^{0}(Y_{k}^{i})(x)=\pm\sqrt{2(k+\frac{m}{2}+2N)}P_{N}^{(0,k+\frac{m}{2}-1)}(2\vert x\vert^{2}-1)Y_{k}^{i}(x)\\
&&\barC_{2N+1,m}^{0}(Y_{k}^{i})(x)=\pm\sqrt{2(k+\frac{m}{2}+1+2N)}P_{N}^{(0,k+\frac{m}{2})}(2\vert x\vert^{2}-1)xY_{k}^{i}(x).\\
\end{eqnarray*}
\end{Th}	
\begin{Remark}
As we have seen, the radial part of the Clifford-Legendre polynomial 
$C_{n,m}^{0}(Y_{k}^{i})(x)$
 is a Jacobi polynomials of degree $[\frac{n}{2}]$. Since (e.g., see \cite{andrews1999special,driver2011stieltjes}) the degree $n$ Jacobi polynomials have exactly $n$ simple zeros on $[-1,1]$, and we conclude that each Clifford-Legendre polynomial
$C_{n,m}^{0}(Y_{k}^{i})(x)$ will be zero on precisely $n$ distinct spheres of radius $r<1$ centred at the origin.
By appealing to the Sturm-Liouville theory associated with the Jacobi polynomials (e.g., see \cite{christensen2016introduction,al2008sturm}) the radii of the spheres on which the even polynomials $\{C_{2N,m}^0(Y_k^i)\}_{N=0}^\infty$ interlace, as do the radii of the spheres on which the odd polynomials $\{C_{2N+1,m}^0(Y_k^i)\}_{N=0}^\infty$ vanish.
What's not clear is that the even and odd polynomial zero sets interlace, i.e., that the radii of the spheres on which the polynomials $\{C_{n,m}^0(Y_k^i)\}_{n=0}^\infty$ interlace, as the radial parts of these functions are eigenfunctions of different differential operators.
\end{Remark}
The following result on the interlacing of the zeros of Jacobi polynomials comes from \cite{driver2008interlacing}.
\begin{Th}\label{interlacehelpful}
Let 
$n\in \mathbb{N},$
$\alpha\geq -1,\; \beta\geq -1$
and let
\begin{eqnarray*}
&&-1<x_{1}<x_{2}<\cdots <x_{n}<1,\; \textnormal{be the zeros of\;} P_{n}^{(\alpha,\beta)},\\
&&-1<t_{1}<t_{2}<\cdots <t_{n}<1,\; \textnormal{be the zeros of\;} P_{n}^{(\alpha,\beta+t)}\textnormal{and}\\
&&-1<y_{1}<y_{2}<\cdots <y_{n}<1,\; \textnormal{be the zeros of\;} P_{n}^{(\alpha,\beta+2)},
\end{eqnarray*}
where 
$0<t<2.$
Then
$$-1<x_{1}<t_{1}<y_{1}<x_{2}<t_{2}<y_{2}<\cdots<x_{n}<t_{n}<y_{n}<1.$$
\end{Th}
\begin{proof}
For the proof see \cite{driver2008interlacing}.
\end{proof}

\begin{Th}\label{Theorem5.2}
The radii of the spheres forming the zero sets of the Clifford-Legendre polynomials $C_{n,m}^0(Y_k^i)(x)$ are interlaced.
\end{Th}
\begin{proof}
First note that by Theorem \ref{interlacehelpful} with $t=1$,  the zero sets of $P_n^{(0,k+\frac{m}{2}-1)}(t)$, $P_n^{(0,k+\frac{m}{2})}(t)$ and $P_{n+1}^{(0,k+\frac{m}{2}-1)}(t)$ on $[-1,1]$ are interlaced. By Theorem \ref{thm: CL and Jacobi}, the radii of the zero sets of 
$C_{2n,m}^0(Y_k^i)$, $C_{2n+1,m}^0(Y_k^i)$ and $C_{2n+2,m}^0(Y_k^i)$ are interlaced. 
\end{proof}

\begin{figure}[h]\label{Figure45}
	\centering
	\includegraphics[width=0.5\linewidth]{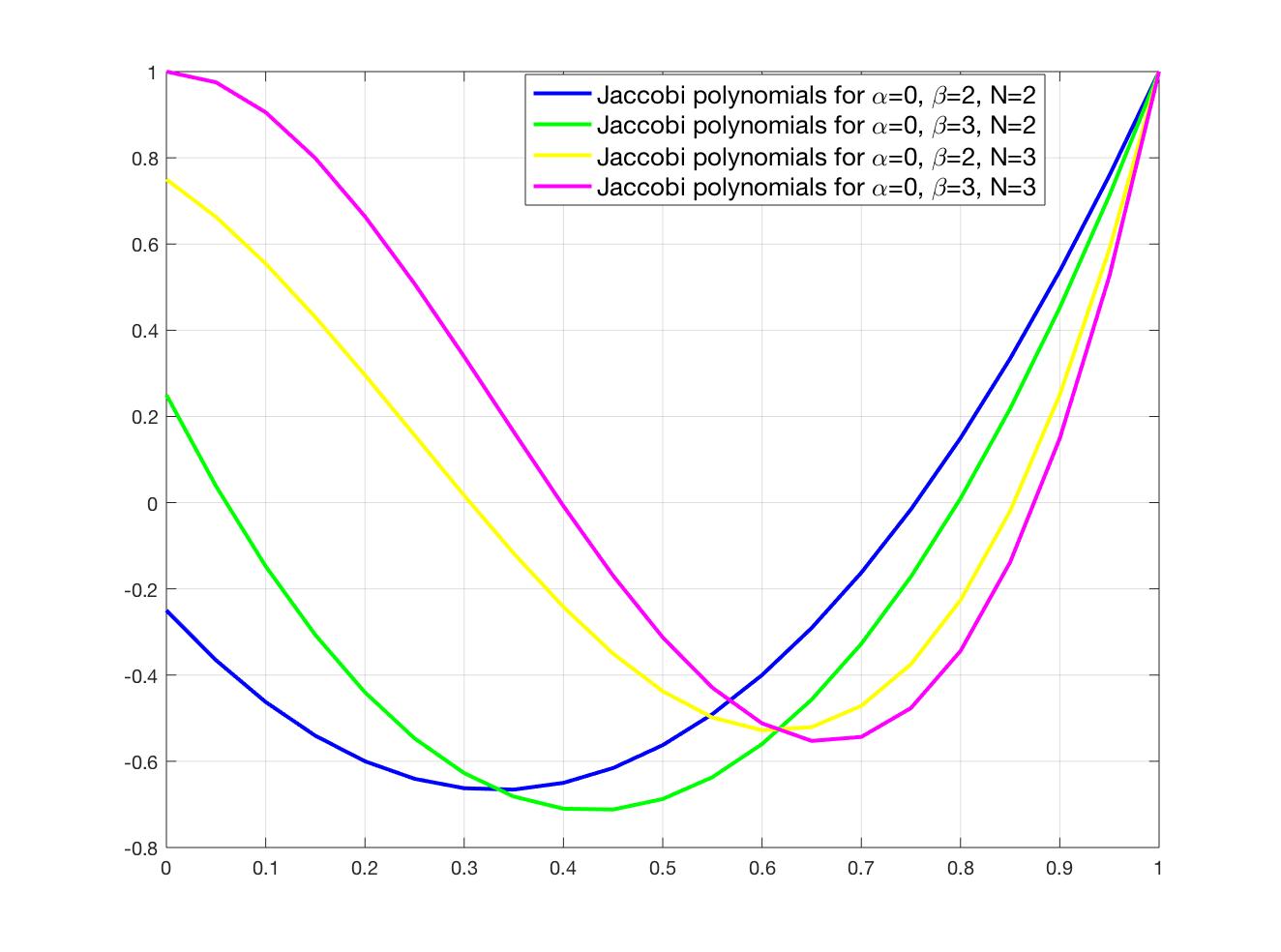}
	\caption{{\small Classic Jacobi polynomials on the real line}}
\end{figure}
%

\section{Bonnet type Formula for the Clifford-Legendre Polynomials}
In this final section, we will prove the Bonnet type formula which expresses $xC_{n,m}^0(Y_k^i)(x)$ as a linear combination of  $C_{n-1,m}^0(Y_k^i)(x)$ and $C_{n+1,m}^0(Y_k^i)(x)$. This is the main motivation for this work, as it allows for the efficient computation of Clifford-prolate functions in higher dimensions \cite{ghaffari2019clifford}.
\begin{Lemma}\label{lem: CL og}
The even Clifford-Legendre polynomials $C_{2n,m}^0(Y_k^i)(x)$ defined on ${\mathbb R}^m$ satisfy
$$\int\limits_{\overline{B(1)}}\overline{C_{2n'}^{0}(Y_{k'}^{i'})(x)}xC_{2n}^{0}(Y_{k}^{i})(x)\, dx=0,$$
for all 
$n, n', k,k',i,i'.$
\end{Lemma} 

\begin{proof} We use the decomposition (\ref{radial decomp}) of the Clifford-Legendre polynomials as products of their radial and tangential parts to obtain
\begin{eqnarray*}
&&\int\limits_{B(1)}\overline{C_{2n'}^{0}(Y_{k'}^{i'})(x)}xC_{2n}^{0}(Y_{k}^{i})(x)dx=\int\limits_{B(1)}\overline{P_{n',k',m}(\vert x\vert^{2})Y_{k'}^{i'}(x)}xP_{n,k,m}(\vert x\vert^{2})Y_{k}^{i}(x)\, dx\hspace*{10cm}\\
&=&\int\limits_{0}^{1}P_{n',k',m}(r^{2})P_{n,k,m}(r^{2})r^{k+k'+2}\, dr\bigg(\int\limits_{S^{m-1}}\overline{Y_{k'}^{j'}(\theta)}\theta Y_{k}^{i}(\theta)\, d\theta\bigg)=0,
\end{eqnarray*}
because of the Clifford-Stokes theorem \ref{Clifford-Stokes theorem} and the monogenicity of $Y_k^i$ and $Y_{k'}^{i'}$.
\end{proof}

\begin{Lemma}\label{lem: int by parts}
For the Clifford-Legendre polynomials $C_{n,m}^0(Y_k^i)$ satisfy the integration by parts formulae
\begin{align*}
\int\limits_{B(1)}\overline{\partial_{x}^{2n}[xC_{2n'+1,m}^{0}(Y_{k}^{i})(x)]}[(1-\vert x\vert^{2})^{2n}Y_{k}^{i}(x)]\, dx&=-\int\limits_{B(1)}\overline{C_{2n'+1,m}^{0}(Y_{k}^{i})(x)} x C_{2n,m}^{0}(Y_{k}^{i})(x)\, dx\\
\int\limits_{B(1)}\overline{\partial_{x}^{2n+1}[xC_{2n',m}^{0}(Y_{k}^{i})(x)]}[(1-\vert x\vert^{2})^{2n+1}Y_{k}^{i}(x)]\, dx&=-\int\limits_{B(1)}\overline{C_{2n',m}^{0}(Y_{k}^{i})(x)} x C_{2n+1,m}^{0}(Y_{k}^{i})(x)\, dx.
\end{align*}
\end{Lemma}
\begin{proof}
We will prove only the first of these formulae as the proof of the second is similar. Note that an application of the Rodrigues' formula (\ref{rodrigues}) and the Clifford-Stokes theorem gives
\begin{align}
&\int\limits_{B(1)}\overline{C_{2n'+1,m}^{0}(Y_{k}^{i})(x)} x C_{2n,m}^{0}(Y_{k}^{i})(x)\, dx=-\int\limits_{B(1)}\overline{xC_{2n'+1,m}^{0}(Y_{k}^{i})(x)}\partial_{x}^{2n}[(1-\vert x\vert^{2})^{2n}Y_{k}^{i}(x)]\, dx\notag\\
&\qquad\qquad=-\int\limits_{S^{m-1}}\overline{xC_{2n'+1,m}^{0}(Y_{k}^{i})(x)}x\partial_{x}^{2n-1}[(1-\vert x\vert^{2})^{2n}Y_{k}^{i}(x)]\, d\sigma (x)\notag\\
&\qquad\qquad+\int\limits_{B(1)}(\overline{xC_{2n'+1,m}^{0}(Y_{k}^{i})(x)}\partial_{x})\partial_{x}^{2n-1}[(1-\vert x\vert^{2})^{2n}Y_{k}^{i}(x)]\, dx\notag\\
&\qquad\qquad=-\int\limits_{B(1)}\overline{\partial_{x}[xC_{2n'+1,m}^{0}(Y_{k}^{i})(x)]}(\partial_{x}^{2n-1}[(1-\vert x\vert^{2})^{2n}Y_{k}^{i}(x)])\,dx\label{int by parts}
\end{align}
since the integrand of the integral over the unit sphere vanishes. The result follows by repeated application of (\ref{int by parts}).\end{proof}
%
%

In the next result, we show that $xC_{n,m}^0(Y_k)(x)$ can be written as a linear combination $C_{n-1,m}^0(Y_k^i)(x)$ and $C_{n+1,m}^0(Y_k^i)(x)$. 

\begin{Lemma}\label{lem: Bonnet part 1}
There exist real constants $\{a_{n,k,i};\ n\geq 0,\ k\geq 0,\ 1\leq i\leq d_k\}$ and $\{b_{n,k,i};\ n\geq 0,\ k\geq 0,\ 1\leq i\leq d_k\}$ such that 
$$xC_{n}^{0}(Y_{k}^{i})(x)=C_{n+1,m}^0(Y_k^i)(x)a_{n,k,i}+C_{n-1,m}^0(Y_k^i)(x)b_{n,k,i}.$$
\end{Lemma}

\begin{proof}
Since $\{C_{n,m}^0(Y_k^i);\ n\geq 0,\ k\geq 0,\ 1\leq i\leq d_k\}$ forms an orthonormal basis for $L^2(B(1),{\mathbb R}_m)$ there are Clifford constants $\{a_{n,k,i}\}$ and $\{b_{n,k,i}\}$ such that 
\begin{align}
xC_{2n,m}^{0}(Y_{k}^{i})(x)&=\sum\limits_{n'}\sum\limits_{i'}\sum\limits_{k'}C_{2n'+1,m}^{0}(Y_{k'}^{i'})(x)\, a_{n',k',i'}\notag\\
&+\sum\limits_{n'}\sum\limits_{i'}\sum\limits_{k'}C_{2n',m}^{0}(Y_{k'}^{i'})(x)\, b_{n',k',i'}.\label{xCL exp}
\end{align}
We multiply both sides of (\ref{xCL exp}) by
$\overline{C_{2M,m}^{0}(Y_{k''}^{i''})(x)}$
from left side and integrate over 
$B(1)$ and apply Lemma \ref{lem: CL og} to find
\begin{eqnarray*}
0&=&\int\limits_{B(1)}\overline{C_{2M,m}^{0}(Y_{k''}^{i''})(x)} x C_{2n,m}^{0}(Y_{k}^{i})(x)dx\\
&=&\sum\limits_{n',i',k'}\bigg( \int\limits_{B(1)}\overline{C_{2M,m}^{0}(Y_{k''}^{i''})(x)}C_{2n'+1,m}^{0}(Y_{k'}^{i'})(x)dx \bigg)a_{n',k',i'}\\
&+&\sum\limits_{n',i',k'}\bigg( \int\limits_{B(1)}\overline{C_{2M,m}^{0}(Y_{k''}^{i''})(x)}C_{2n',m}^{0}(Y_{k'}^{i'})(x)dx \bigg)b_{n',k',i'}\\
&=&\int\limits_{B(1)} \vert C_{2M,m}^{0}(Y_{k''}^{i''})(x)\vert^{2}b_{M,k'',i''}
\end{eqnarray*}
where in the last step we have used the orthogonality of the Clifford-Legendre polynomials on $B(1)$.
 We conclude that $b_{M,k'',i''}=0$ and therefore (\ref{xCL exp}) simplifies to
\begin{equation}
xC_{2n,m}^{0}(Y_{k}^{i})(x)=\sum\limits_{n'}\sum\limits_{i'}\sum\limits_{k'}C_{2n'+1,m}^{0}(Y_{k'}^{i'})(x)a_{n',k',i'}.\label{xCL exp2}
\end{equation}
We recall the radial decompositions 
$$C_{2n,m}^0(Y_k^i)(x)=P_{n,k,m}(|x|^2)Y_k^i(x);\qquad C_{2n+1,m}^0(Y_k^i)(x)=xQ_{n,k,m}(|x|^2)Y_k^i(x)$$
 and apply them to (\ref{xCL exp2}) to obtain 
\begin{equation}
P_{n,k,m}(r^{2})r^{k}Y_{k'}^{i'}(\omega)=\sum\limits_{n',i',k'}Q_{n',k',m}(r^2)r^{k'}Y_{k'}^{i'}(\omega)a_{n',k',i'}\label{xCL exp3}
\end{equation}
with $r>0$ and $\omega\in S^{m-1}$.
Multiplying both sides of this equation by 
$\overline{Y_{k''}^{i''}(\omega)}$
and integrating over 
$S^{m-1}$
yields
$$P_{n,k,m}(r^{2})r^{k}\int\limits_{S^{m-1}}\overline{Y_{k''}^{i''}(\omega)}Y_{k}^{i}(\omega)\, d\omega =\sum\limits_{n',i',k'}Q_{n',k',m}(r^2)r^{k'}\int\limits_{S^{m-1}}\overline{Y_{k''}^{i''}(\omega)}Y_{k'}^{i'}(\omega)\, d\omega\, a_{n',k',i'},$$
and the orthogonality of $\{Y_k^i;\ k\geq 0,\ 1\leq i\leq d_k\}$ on $S^{m-1}$ gives 
$$P_{n,k,m}(r^{2})r^{k}\delta_{i'',i}\delta_{k'',k}=\sum\limits_{n',i',k'}Q_{n',k',m}(r^2)r^{k'}\delta_{i'',i'}\delta_{k'',k'}\; a_{n',k',i'}=Q_{n',k'',m}(r^{2})r^{k''}\, a_{n',k',i'},$$
We conclude that $a_{n',k',i'}=0$ unless $i=i'=i''$ and $k=k'=k''$. As a consequence, (\ref{xCL exp3}) simplifies to
%
\begin{equation}
P_{n,k,m}(t)=\sum_{n'=0}^{\infty}Q_{n',k,m}(t)\, a_{n',k,i}.\label{xCL exp4}
\end{equation}
Let $w_{k,m}(t)=2t^{k+\frac{m}{2}}$. We multiply both sides by 
$w_{k,m}(t)Q_{n'',k}(t)$
and integrate over $[0,1]$ to obtain
\begin{align}
\int_{0}^{1}Q_{n'',k,m}(t)P_{n,k,m}(t)w_{k,m}(t)\, dt&=\sum\limits_{n'=0}^{\infty}\bigg(\int_{0}^{1}Q_{n'',k,m}(t)Q_{n',k,m}(t)w_{k,m}(t)\, dt\bigg)a_{n,k,i}\notag\\
&=\Vert Q_{n'',k,m}\Vert^{2}_{w_{k,m}}a_{n'',k,i}.\label{xCL exp5}
\end{align}
However, $P_{N,k,m}(t)=c_{N,k,m}P_N^{(0,k+\frac{m}{2}-1)}(2t-1)$ and $Q_{N,k,m}(t)=d_{N,k,m}P_N^{(0,k+\frac{m}{2}-1)}(2t-1)$, so the left hand side of (\ref{xCL exp5}) becomes
\begin{equation}
\int\limits_{0}^{1}Q_{n'',k,m}(t)P_{n,k,m}(t)w_{k,m}(t)\, dt=c_{n,k,m}d_{n,k,m}\int_{-1}^1P_{n''}^{(0,k+\frac{m}{2})}(s)P_n^{(0,k+\frac{m}{2}-1)}(s)(s+1)^{k+\frac{m}{2}}\, ds .\label{Jacobi og}
\end{equation}
However, The Jacobi polynomial $P_{n''}^{(0,k+\frac{m}{2})}$ has degree $n''$ and is orthogonal to all polynomials of lower degree when the inner product on $[-1,1]$ is computed relative to the weight function $(s+1)^{k+\frac{m}{2}}$. We conclude from (\ref{Jacobi og}) that the left hand side of (\ref{xCL exp5}) is zero when $n<n''$
and hence that $a_{n'',k,i}=0$ for $n''>n$. Hence, we have
\begin{equation}
xC_{2n,m}^{0}(Y_{k}^{i})(x)=\sum\limits_{n'=0}^{n}C_{2n'+1,m}^{0}(Y_{k}^{i})(x)\, a_{n',k,i}.\label{xCL exp6}
\end{equation}
Multiplying both sides of (\ref{xCL exp6}) on the left by 
$\overline{C_{2M+1,m}^{0}(Y_{k}^{i})(x)}$ $(0\leq M\leq n)$ and integrating over $B(1)$ gives, with an application of the first part of Lemma \ref{lem: int by parts},
\begin{equation}
-\int_{B(1)}\overline{\partial_{x}^{2n}[xC_{2M+1,m}^{0}(Y_{k}^{i})(x)]}(1-\vert x\vert^{2})^{2n}Y_{k}^{i}(x)dx=\Vert C^{0}_{2M+1,m}(Y_{k}^{i})\Vert^{2}_{L^2(B(1))}\, a_{M,k,i}.\label{xCL exp7}
\end{equation}
If $F_n$, $G_n$ are polynomials of degree $n$ and $Y_k$ is a spherical monogenic of degree $k$, a simple calculation shows that 
$$\partial_x(F_n(|x|^2)Y_k(x))=xH_{n-1}(|x|^2)Y_k(x);\quad \partial_x(xG_n(|x|^2)Y_k(x))=I_{n}(|x|^2)Y_k(x)$$
where $H_{n-1}$ and $I_n$ are polynomials of degree $n-1$ and $n$ respectively. We conclude that the left hand side of (\ref{xCL exp7})
is zero  when 
$2n>2M+2$
so that $a_{M,k,i}=0$ for $M<n-1$.
We conclude from (\ref{xCL exp6}) that 
$$xC_{2n,m}^{0}(Y_{k}^{i})(x)=\sum\limits_{n'=n-1}^{n}C_{2n'+1,m}^{0}(Y_{k}^{i})(x)\, a_{n',k,i}$$
as required for the case of even Clifford-Legendre polynomials. A similar argument, using second part of Lemma \ref{int by parts}, deals with the odd case.
\end{proof}	
\begin{Th}\label{Bonnet Formula}
The Bonnet type Formula for the Clifford-Legendre polynomials is as follows:
	\begin{itemize}
		\item[(a)]
		when
		$n$
		is odd,	
		$$xC_{2N+1,m}^{0}(Y_{k}^{i})(x)=\alpha_{N,k}C_{2N+2,m}^{0}(Y_{k}^{i})(x)+\beta_{N,k}C_{2N,m}^{0}(Y_{k}^{i})(x),$$
		where 
		$$\alpha_{N,k}=\frac{-1}{4(\frac{m}{2}+2N+k+1)};\quad \beta_{N,k}=\frac{2(2N+1)(\frac{m}{2}+N+k)}{(\frac{m}{2}+2N+k+1)}.$$
%
		\item[(b)]
		when
		$n$
		is even,
		$$xC_{2N,m}^{0}(Y_{k}^{i})(x)=\alpha'_{N,k}C_{2N+1,m}^{0}(Y_{k}^{i})(x)+\beta'_{N,k}C_{2N-1,m}^{0}(Y_{k}^{i})(x),$$
		where
		$$\alpha'_{N,k}=\frac{-(\frac{m}{2}+N+k)}{2(2N+1)(\frac{m}{2}+2N+k)},\quad\beta'_{N,k}=\frac{4N^{2}}{(\frac{m}{2}+2N+k)}.$$
	\end{itemize}
\end{Th}
\begin{proof}
	By Lemma \ref{lem: Bonnet part 1}, when $n=2N+1$ is odd, there are Clifford constants $\alpha_{N,k,i}$ and $\beta_{N,k,i}$ for which 	
\begin{equation}\label{equation3.13new}
	xC_{2N+1,m}^{0}(Y_{k}^{i})(x)=C_{2N+2,m}^{0}(Y_{k}^{i})(x)\alpha_{N,k,i}+C_{2N,m}^{0}(Y_{k}^{i})(x)\beta_{N,k,i}.
	\end{equation}
	Now, by the explicit representation of Clifford-Legendre polynomials given in Theorem \ref{thm: representationofCl-Legendre}, equation (\ref{equation3.13new}) may be written as 
%
\begin{align*}
&-\frac{2^{2N+1}(2N+1)!}{N!}\sum_{j=0}^{N}{N\choose j}\frac{\Gamma(j+k+\frac{m}{2}+N+1)}{\Gamma(j+k+\frac{m}{2}+1)}(-1)^{j}\vert x\vert^{2j+2}\\
&=\bigg[\frac{2^{2N+2}(2N+2)!}{(N+1)!}\sum\limits_{j=0}^{N+1}{N+1\choose j}\frac{\Gamma(j+k+\frac{m}{2}+N+1)}{\Gamma(j+k+\frac{m}{2})}(-1)^{j}\vert x\vert^{2j}\bigg]\alpha_{N,k,i}\\
&+\bigg[\frac{2^{2N}(2N)!}{N!}\sum\limits_{j=0}^{N}{N\choose j}\frac{\Gamma(j+k+\frac{m}{2}+N)}{\Gamma(j+k+\frac{m}{2})}(-1)^{j}\vert x\vert^{2j}\bigg]\beta_{N,k,i}.
\end{align*}
 We now equate coefficients of the powers of $|x|^2$ on both sides of this equation.  By equating coefficients of $|x|^{2N+2}$, we find
\begin{align*}
&\frac{2^{2N+2}(2N+2)!}{(N+1)!} \frac{\Gamma(k+\frac{m}{2}+2N+2)}{\Gamma(N+k+\frac{m}{2}+1)}(-1)^{N+1}\alpha_{N,k,i}\\
&=\frac{2^{2N+1}(2N+1)!}{N!} \frac{\Gamma(k+\frac{m}{2}+2N+1)}{\Gamma(N+k+\frac{m}{2}+1)}(-1)^{N}
\end{align*}
from which we conclude that $\alpha_{N,k,i}=\alpha_{N,k}=\frac{-1}{4(\frac{m}{2}+2N+k+1)}$. Similarly, by equating the constant terms, we find that 
	$$\bigg[\frac{2^{2N+2}(2N+2)!}{(N+1)!}\frac{\Gamma(k+\frac{m}{2}+N+1)}{\Gamma(k+\frac{m}{2})}\bigg]\alpha_{N,k}+\bigg[\frac{2^{2N}(2N)!}{N!} \frac{\Gamma(k+\frac{m}{2}+N)}{\Gamma(k+\frac{m}{2})}\bigg]\beta_{N,k,i}=0.$$
	By replacing
	$\alpha_{N,k}$ by its known value gives
	$\beta_{N,k,i}=\beta_{N,k}=\frac{2(2N+1)(\frac{m}{2}+N+k)}{(\frac{m}{2}+2N+k+1)}$. This completes the proof of part (a) of the Theorem.

	Part (b) is proved in a similar manner.	
\end{proof}

When applied to the construction of Clifford-prolate functions in higher dimensions, the Bonnet type formula is required for normalized Clifford-Legendre polynomials. In that context, the Bonnet type formula takes the form below.

\begin{Corollary}
The Bonnet type Formula formula for the normalized Clifford-Legendre polynomials is as follows:
\begin{enumerate}
\item[(a)] When $n=2N$ is even,
$$x\barC_{2N,m}^0(Y_k^i)(x)=A_{N,k,m}\barC_{2N+1,m}^0(Y_k^i)(x)+B_{N,k,m}\barC_{2N-1,m}^0(Y_k^i)(x)$$
where
\begin{align*}
A_{N,k,m}&=\frac{-(\frac{m}{2}+N+k)\sqrt{m+4N+2k}}{(\frac{m}{2}+2N+k)\sqrt{m+4N+2k+2}}\\
B_{N,k,m}&=\frac{N\sqrt{m+4N+2k}}{(\frac{m}{2}+2N+k)\sqrt{m+4N+2k-2}}
\end{align*}
\item[(b)] When $n=2N+1$ is odd
$$x\barC_{2N+1,n}^0(Y_k^i)(x)=A_{N,k,m}'\barC_{2N+2,m}^0(Y_k^i)(x)+B_{N,k,m}'\barC_{2N,m}^0(Y_k)(x)$$
where
\begin{align*}
A_{N,k,m}'&=\frac{-(N+1)\sqrt{m+4N+2k+2}}{(\frac{m}{2}+2N+k+1)\sqrt{m+4N+2k+4}}\\
B_{N,k,m}'&=\frac{(\frac{m}{2}+N+k)\sqrt{m+4N+2k+2}}{(\frac{m}{2}+2N+k+1)\sqrt{m+4N+2k}}.
\end{align*}
\end{enumerate}
\end{Corollary}

\section*{Acknowledgment}
\noindent The authors would like to thank the Center for Computer-Assisted Research in Mathematics and its Applications at the University of Newcastle for its continued support. JAH is supported by the Australian Research Council through Discovery Grant DP160101537. Thanks Roy. Thanks HG.
\bibliographystyle{siam}
\bibliography{Clifford_Legendre_Polynomial_arXiev}

\end{document}